\documentclass[12pt,reqno,amscd,amssymb,pstricks,latexsym,amsbsy,xypic,mathrsfs,verbatim]{amsart}
\usepackage{amsfonts,amssymb,amsthm}
\usepackage{amsmath,amscd}
\usepackage{pstricks}
\usepackage{pstricks,pst-node}
\usepackage{mathrsfs}
\usepackage[all]{xy}
\usepackage[enableskew,vcentermath]{youngtab}

\renewcommand{\subsection}[1]{\vspace{.18in}\par\noindent\addtocounter{subsection}{1}\setcounter{equation}{0}{\bf\thesubsection.\hspace{5pt}#1}}

\newtheorem{theorem}{Theorem}[section]
\numberwithin{equation}{theorem}

\theoremstyle{definition}
\newtheorem{Def}[theorem]{Definition}%[section]

\newtheorem{Example}[theorem]{Example}%[section]
\newtheorem{Rem}[theorem]{Remark}

\newtheorem{Rems}[theorem]{Remarks}
\theoremstyle{plain}
\newtheorem{Prop}[theorem]{Proposition}
\newtheorem{Thm}[theorem]{Theorem}

\newtheorem{Lem}[theorem]{Lemma}
\newtheorem{Coro}[theorem]{Corollary}

\def\sB{{\mathcal B}}

\def\sH{{\mathcal H}}
\def\sI{{\mathcal I}}
\def\sJ{{\mathcal J}}

\def\sS{{\mathcal S}}
\def\sT{{\mathcal T}}

\def\sX{{\mathcal X}}

\def\sZ{{\mathcal Z}}

\newcommand{\scr}[1]{\mathscr #1}

\def\fS{{\frak S}}

\def\sfc{{\mathsf c}}

\def\sfz{{\mathsf z}}
\def\sfs{{\mathsf s}}

 \newcommand{\bfOg}{{\bf\Omega}}

\newcommand{\msD}{\mathscr D}

\newcommand{\Hr}{{\sH(r)}}

\newcommand{\vtg}{{\!\vartriangle\!}}
\newcommand{\Hall}{{{\mathfrak H}_\vtg(n)}}

\newcommand{\bfHall}{{\boldsymbol{\mathfrak H}_\vtg(n)}}
\newcommand{\Hallpi}{\bfHall^{\geq 0}}
\newcommand{\Hallmi}{\bfHall^{\leq 0}}

\newcommand{\dHallr}{{\boldsymbol{\mathfrak D}_\vtg}(n)}

\def\ggp#1#2{\left[\kern-3.2pt\left[{#1\atop #2}\right]\kern-3.2pt\right]}

\newcommand{\leb}{\left[}

\newcommand{\rib}{\right]}

\newcommand{\row}{\text{ro}}

\newcommand{\Lanr}{\Lambda(n,r)}

\newcommand{\tri}{\triangle(n)}

\newcommand{\afsl}{\widehat{\frak{sl}}_n}

\newcommand{\afgl}{\widehat{\frak{gl}}_n}

\def\leq{\leqslant}\def\geq{\geqslant}

 \newcommand{\lm}{\longmapsto}

  \newcommand{\Og}{\Omega}

 \newcommand{\og}{\omega}

 \newcommand{\up}{v}

 \newcommand{\ep}{\epsilon}
 \newcommand{\al}{\alpha}

 \newcommand{\ti}{\widetilde}

\def\th{\theta}

\newcommand{\ot}{\otimes}

\newcommand{\bfe}{\mathbf{e}}

\newcommand{\lra}{\longrightarrow}
\newcommand{\ra}{\rightarrow}
 \newcommand{\la}{{\lambda}}

 \newcommand{\La}{\Lambda}

 \newcommand{\mbn}{\mathbb N}
 \newcommand{\mbq}{\mathbb Q}
 
 \newcommand{\mbz}{\mathbb Z}

 \newcommand{\bfs}{{\mathbf{s}}}

\newcommand{\bfa}{{\mathbf{a}}}
\newcommand{\bfb}{{\mathbf{b}}}
\newcommand{\bfm}{{\mathbf{m}}}
 \newcommand{\bfu}{{\mathbf{m}}}
 \newcommand{\bfU}{{\mathbf{U}}}

\newcommand{\bfc}{{\mathbf{c}}}

\newcommand{\End}{\operatorname{End}}
\newcommand{\Hom}{\operatorname{Hom}}

\newcommand{\diag}{\operatorname{diag}}

\def\ro{\text{\rm ro}}
\def\co{\text{\rm co}}

\def\fD{{\mathfrak D}}

\def\id{{\operatorname{id}}}

\newcommand{\scR}{{\mathscr R}}
\newcommand{\scK}{{\mathscr K}}

\newcommand{\afsHr}{{{\mathcal H}_\vtg(r)}}
\newcommand{\afsSr}{{{\mathcal S}}_\vtg(n,r)}
\newcommand{\cysHr}{{{\mathcal H}_\bfu(r)}}
\newcommand{\tbsHr}{{{\mathcal H}_{\bf 2}(r)}}

\newcommand{\cysSr}{{{\mathcal S}}_\bfu(n,r)}
\newcommand{\tbsSr}{{{\mathcal S}}_{\bf2}(n,r)}
\newcommand{\afsHrR}{{{\mathcal H}_{\vtg}(r)_{\scR}}}
\newcommand{\afsSrR}{{{\mathcal S}}_{\vtg}(n,r)_{\scR}}
\newcommand{\cysHrR}{{{\mathcal H}_{\bfu}(r)_{\scR}}}
\newcommand{\busHrR}{{{\mathcal H}_{\bullet}(r)_{\scR}}}
\newcommand{\cysSrR}{{{\mathcal S}}_{\bfu}(n,r)_{\scR}}
\def\la{\lambda}
\def\da{\mathbb{A}}
\def\cnrm{{\Theta}_m(n,r)}
\def\row{\text{ro}}
\def\col{\text{co}}
\def\dlm{\msD_{\lambda,\mu}}
\newcommand{\Si}{\mathbf{S}^{\imath}}
\def\wdl{\widetilde{\msD}_\lambda}
\def\wdlm{\widetilde{\msD}_{\lambda,\mu}}

\begin{document}

\title{Slim cyclotomic $q$-Schur algebras}

\author{Bangming Deng, Jie Du and Guiyu Yang}
\address{Yau Mathematical Sciences Center, Tsinghua University,
Beijing 100084,  China.} \email{ bmdeng@math.tsinghua.edu.cn }
\address{School of Mathematics and Statistics, University of New South Wales,
Sydney, NSW 2052, Australia.} \email{j.du@unsw.edu.au\quad{\it
Home Page: {\tt http://www.maths.unsw.edu.au/$\sim$jied}}}
\address{School of Mathematics and Statistics, Shandong University of Technology,
Zibo, 255000, China.} \email{yanggy@mail.bnu.edu.cn}

%\thanks{Supported partially by the Australian Research Council, the Natural
%Science Foundation of China, and the Doctoral Program of Chinese Higher%
%Education. The research was carried out while Deng was visiting
%the University of New South Wales. The hospitality and support of
%UNSW are gratefully acknowledged.}

\date{\today}

\thanks{This work was supported by a 2017 UNSW Science Goldstar Grant (Jie Du), and National Natural Science Foundation of China (
Grant Nos. 11671024, 11671234 and 11331006).}

\begin{abstract}We construct a new basis for a slim cyclotomic $q$-Schur algebra $\cysSr$ via symmetric
polynomials in Jucys--Murphy operators of the cyclotomic Hecke algebra $\cysHr$. We show that this basis,
labelled by matrices, is not the double coset basis when $\cysHr$ is the Hecke algebra of a Coxeter group,
but coincides with the double coset basis for the corresponding group algebra, the Hecke algebra at $q=1$.
As further applications, we then discuss the cyclotomic Schur--Weyl duality at the integral level.
This also includes a category equivalence and a classification of simple objects.
\end{abstract}

%\tableofcontents
\maketitle

\section{Introduction}

Cyclotomic $q$-Schur algebras were first introduced  by Dipper, James and Mathas (DJM) \cite{DJM}
as a natural generalisation of Dipper and James' $q$-Schur algebra  \cite{DJ89} associated with
the symmetric groups $\fS_r$ to a similar algebra associated with the wreath products
$\mathbb Z_m\wr\fS_r$ of the cyclic group $\mathbb Z_m$ with $\fS_r$. See also \cite{DS}
for the $q$-Schur$^{2\textsc{b}}$ (or $q$-Schur$^2$) algebra associated with the type $B$
Weyl group  $\mathbb Z_2\wr\fS_r$ and \cite{Gr97, VV, DDF} for the affine $q$-Schur algebra
associated with the affine symmetric group $\mathbb Z\wr\fS_r$. Like the $q$-Schur algebras,
the cyclotomic ones are quasi-hereditary with a cellular basis; see \cite[(6.12), (6.18)]{DJM}.

 As an endomorphism algebra of the direct sum of  certain cyclic modules over the cyclotomic
 Hecke algebra \cite{AK}, the construction of a cyclotomic $q$-Schur algebra involves {\it all}
 $m$-fold multipartitions of $r$,  which index all irreducible characters of the finite group
 $\mathbb Z_m\wr\fS_r$.  However, when the ``cyclotomic Schur--Weyl duality'' is under consideration,
 these algebras seem a bit too ``fat''. Thus, as suggested in \cite{LR},  a certain centraliser algebra,
 involving only partitions of $r$, has been considered to play such a role. We will call these
 centraliser algebras {\it slim cyclotomic $q$-Schur algebras} in this paper.

A further study on slim cyclotomic $q$-Schur algebras is motivated by the following recent works:
First, the investigation on affine $q$-Schur algebras and affine Schur--Weyl duality has made significant
progress. This includes a full generalisation to the affine case \cite{DDF, DF} of the original work \cite{BLM}
of Beilinson--Lusztig--MacPherson (BLM) and a classification of irreducible representations of affine
$q$-Schur algebras. These works should have applications to  slim cyclotomic $q$-Schur algebras. Second,
there has been also a breakthrough in generalising the BLM work to types other than $A$ series by
Bao--Kujawa--Li--Wang \cite{BKLW} for type B/C; see also \cite{FL} for type $D$ and \cite{FLLW}
for affine type $C$. The $q$-Schur algebra ${\bf S}^\iota$ defined in \cite[\S6.1]{BKLW} is used in
linking the Hecke algebra of type $C$ with a certain quantum symmetric pair, where a coideal subalgebra
is constructed to play the role required in the Schur--Weyl duality in this case. The algebra ${\bf S}^\iota$ is
the $m=2$ slim cyclotomic $q$-Schur algebra. This algebra is also called the hyperoctahedral Schur
algebra in \cite{Gr97}  or the $q$-Schur$^{1\textsc b}$ algebra in \cite{DS1}.

As a centraliser algebra of DJM's cyclotomic $q$-Schur algebra, a slim cyclotomic $q$-Schur algebra
inherits a cellular basis that does not provide enough information for the representations of the algebra.
However, the most natural basis, analogous to the well-known orbital basis  for the endomorphism algebra
of a permutation module (see \cite{S}), is not available if the complex reflection groups $\mathbb Z_m\wr\fS_r$
is not a Coxeter group. In this paper, we will reveal a new integral basis for a slim cyclotomic
$q$-Schur algebra. This basis is different from the usual double coset basis for the Hecke endomorphism
algebra associated with a Coxeter group. However, it still reduces to a double coset basis for
the group algebra. So we may regard our new basis as a new quantisation of the double coset basis,
generalising the natural basis for the endomorphism algebra of a permutation module considered in
\cite{S} to this cyclotomic world.

It is interesting to point out that the definition of this new basis involves symmetric polynomials
in Jucys--Murphy operators and can be lifted to affine $q$-Schur algebras. In this way, we obtain
an algebra epimorphism from an affine $q$-Schur algebra to a slim cyclotomic $q$-Schur algebra of
the same degree. Thus, an epimorphism from the integral quantum loop algebra of $\mathfrak{gl}_n$
to a slim cyclotomic $q$-Schur algebra can be constructed. Thus, we may put slim cyclotomic $q$-Schur algebras
in the context of a possible cyclotomic Schur--Weyl duality. In fact,
we propose a candidate for a possible ``finite type'' quantum subalgebra which can be used to replace
$\bfU(\afgl)$ in the cyclotomic Schur--Weyl duality. We will also establish a Morita equivalence
in the cyclotomic case and present a classification of irreducible modules for a slim cyclotomic
$q$-Schur algebra. It should be noted that the Schur--Weyl duality for cyclotomic Hecke algebras
has been investigated by many authors; see, for example, \cite{ATY, SS,Hu}. However, the tensor
spaces considered in these works are different from those used in the current paper.

In a forthcoming paper \cite{DDY}, we will determine which slim cyclotomic $q$-Schur algebras
are quasi-hereditary and further investigate the representations of such an algebra at roots of unity.

We organize the paper as follows. In Section 2, we give a brief background on
cyclotomic Hecke algebras and slim cyclotomic $q$-Schur algebras. In Section 3, we will see
how the symmetric polynomials in Jucys--Murphy operators are used in the construction of an integral
basis for $\sS_\bfu(1,r)$, which is a centraliser algebra of $\sS_\bfu(n,r)$ associated with the partition $(r)$,
and prove that it is commutative. In Section 4, we generalise the construction to the entire algebra
$\sS_\bfu(n,r)$ by employing Mak's double coset description in term of matrices \cite{Mak}.
In Section 5, we will make a comparison of our new basis with
the usual double coset basis for the Coxeter group, $\fS_{2,r}$ of type $B$. Sections 6 and 7 are devoted
to the cyclotomic Schur--Weyl duality. Partial double centraliser property is established through affine
$q$-Schur algebras and quantum loop algebra of $\mathfrak{gl}_n$. A certain category equivalence has also
been lifted to the cyclotomic case. Finally, as a further application of our discovery, we classify all
simple objects for the slim cyclotomic $q$-Schur algebras as an application of a similar classification
for affine $q$-Schur algebras. The two appendices cover some technical proofs and a generalisation of
Lemma \ref{sym} to the affine case.

\section{Ariki--Koike algebras and slim cyclotomic $q$-Schur algebras}

For any nonnegative integer $m\geq0$, let $\mbz_m=\mbz/m\mbz$.
We always identify $\mbz_m$ with the subset $\mbz_m=\{0,1,\ldots,m-1\}$ of $\mbz$ if $m\geq1$
(and, of course, $\mbz_0=\mbz$). Thus,
$$\mbz_m^r=\{\bfa=(a_i)\in\mbz^r\mid a_i\in\mbz_m,\;\forall\,1\leq i\leq r\}$$
is a subset of $\mbz^r$.

Let $\fS_r=\fS_{\{1,2,\ldots, r\}}$ be the symmetric group on $r$ letters with the basic transpositions
$s_i=(i,i+1)$, $1\leq i\leq r-1$, as generators, and let $\fS_{m,r}=\mbz_m\wr\fS_r$ be the wreath product
of the cyclic group $\mbz_m$ and $\fS_r$. Then $\fS_{0,r}=\fS_{\vtg,r}$ is the affine symmetry group,
$\fS_{1,r}=\fS_r$, $\fS_{2,r}$ is the Weyl group of type $B$, and $\fS_{m,r}$ for $m\geq 3$ are the complex
reflection groups of type $G(m,1,r)$.

From now onwards, $\scR$ always denotes a commutative ring (with $1$) containing invertible element $q$.

For $m\geq 1$ and an $m$-tuple
$$\bfu=(u_1,\ldots,u_m)\in \scR^m,$$ the {\it cyclotomic Hecke algebra} (or the Ariki--Koike algebra) over $\scR$
$$\cysHr=\cysHrR,$$
associated with $\fS_{m,r}$ (and parameters $q,\mathbf m$), is by definition the $\scR$-algebra generated
by $T_i,L_j$ ($1\leq i\leq r-1,1\leq j\leq r$) with defining relations
\begin{itemize}
\item[(CH1)] $(T_i+1)(T_i-q)=0,$
\item[(CH2)] $T_iT_{i+1}T_i=T_{i+1}T_iT_{i+1},\;\;T_iT_j=T_jT_i\;(|i-j|>1),$
\item[(CH3)] $L_iL_j=L_jL_i,$ $(L_1-u_1)\cdots(L_1-u_m)=0,$
\item[(CH4)]  $T_iL_iT_i=q L_{i+1},\;L_jT_i=T_iL_j\;(j\not=i,i+1).$
%$T_0T_1T_0T_1=T_1T_0T_1T_0,\;T_0T_i=T_iT_0\; (i>1),$
%$(T_0-u_1)\cdots(T_0-u_m)=0,$
\end{itemize}

For $m=0$, the affine Hecke algebra
$$\afsHr=\afsHrR,$$
associated with $\fS_{\vtg,r}$ (and parameter $q$), is presented by the
generators $T_i,\;X_j^\pm$ for $1\leq i\leq r-1$ and $1\leq j\leq r$ and relations (CH1), (CH2) together with
\begin{itemize}
\item[(CH3$'$)] $X_iX_j=X_jX_i$, $X_iX_i^{-1}=1=X_i^{-1}X_i$,
\item[(CH4$'$)] $ T_iX_iT_i=q X_{i+1},\;\; X_jT_i=T_iX_j\;(j\not=i,i+1)$.
\end{itemize}
In both cases, the subalgebra generated by $T_1,\ldots, T_{r-1}$
is the Hecke algebra $\Hr=\sH(r)_\scR$ of $\fS_r$.
If {\it all $u_i\in\scR$ are invertible}, then the $L_i$ are invertible in $\cysHr$. Thus, in this case,
there is a natural surjective homomorphism\footnote{When $m=1$, $\ep_\bfu$ is known as the evaluation map.}
\begin{equation}\label{epimor-af-AK-alg}
\ep_\bfu:\afsHr\lra \cysHr,\;T_i\lm T_i,\,X_j\lm L_j,
 \end{equation}
which gives an algebra isomorphism
$$ \afsHr/\langle (X_1-u_1)\cdots(X_1-u_m)\rangle\cong \cysHr.$$
Thus, the algebra $\cysHr$ is a cyclotomic quotient of the affine Hecke algebra $\afsHr$ when all the parameters
$u_i$ are invertible.

The algebra $\cysHr$ admits an anti-automorphism
\begin{equation}\label{tau}
\tau:\cysHr\lra\cysHr,\quad T_i\longmapsto T_i, L_j\longmapsto L_j.
\end{equation}
Since $q$ is invertible in $\scR$, all the $T_i$ for $1\leq i\leq r-1$ are invertible with
$$T_i^{-1}=q^{-1}(T_i-(q-1)).$$
By multiplying $T_i^{-1}$ and $T_i$ on both sides of $T_iL_iT_i=qL_{i+1}$,
respectively, we get the following two equalities
\begin{equation}\label{LiTi}
L_iT_i=T_iL_{i+1}+(1-q)L_{i+1}, \;\;\; L_{i+1}T_i=T_iL_i+(q-1)L_{i+1}.
\end{equation}

 For each $\bfa=(a_1, \ldots, a_r)\in\mbn^r$, write
$$L^\bfa=L_1^{a_1}\cdots L_r^{a_r}\in\cysHr,\quad X^\bfa=X_1^{a_1}\cdots X_r^{a_r}\in\afsHr.$$
 Then, by an inductive argument, we have for each $1\leq i\leq r-1$ and $\bfa\in \mbn^r$,
\begin{equation} \label{Bernstein-formula}
L^\bfa T_i=T_i L^{\bfa s_i}+(1-q)\frac{L_{i+1}(L^\bfa-L^{\bfa s_i})}{L_i-L_{i+1}},
\end{equation}
 where $\bfa s_i$ is given by permuting the $i$-th and $(i+1)$-th entries of $\bfa$.
 The affine version of the formula (with $L^\bfa$ replaced by $X^\bfa$) is well known; see, e.g., \cite[(3.3.0.2)]{DDF}.

 For each $w\in\fS_r$ and $\bfa\in\mathbb Z^r$, define the place permutation
\begin{equation}\label{place}
\bfa w=(a_{w(1)}, a_{w(2)},\cdots, a_{w(r)}).
\end{equation}
 For each $1\leq i\leq r-1$, set
 $$\alpha_i=(0,\ldots, 1, -1, 0, \ldots, 0)\in \mbz^r$$
with $1$ in the $i$-th position. Thus, $\bfa s_i=\bfa+(a_{i+1}-a_i)\alpha_i$.
Applying \eqref{Bernstein-formula} gives the following lemma.

\begin{Lem}\label{commutator}
Suppose that $1\leq i\leq r-1$ and $\bfa\in \mbn^r$. Then in $\cysHr$
$$L^\bfa T_i=\begin{cases}
T_iL^{\bfa}=T_iL^{\bfa s_i},& \text{if $a_i=a_{i+1}$};\\
T_iL^{\bfa s_i}+(q-1)\sum_{t=1}^{a_{i+1}-a_i}L^{\bfa s_i-t\alpha_i},\;\;& \text{if $a_i<a_{i+1}$};\\
T_iL^{\bfa s_i}+(1-q)\sum_{t=0}^{a_{i}-a_{i+1}-1}L^{\bfa s_i+t\alpha_i}\;\;& \text{if $a_i>a_{i+1}$}.
\end{cases}$$
A similar formula holds in $\afsHr$.
\end{Lem}

 The following result is taken from \cite[Thm 3.10]{AK}.

\begin{Lem} \label{pbw basis} The algebra $\cysHr$ is a free $\scR$-module with bases
$$\{T_wL^{\bfa}\mid w\in\fS_r, \bfa\in\mbz_m^r\}\quad\text{and}\quad \{L^{\bfa}T_w\mid w\in\fS_r, \bfa\in\mbz_m^r\}.$$
Replacing $\mbz_m^r$ by $\mbz^r$ and $L^\bfa$ by $X^\bfa$ yields bases for $\afsHr$.
\end{Lem}

A sequence $\la=(\la_1,\ldots,\la_n)$ of non-negative integers is called a composition (resp., partition)
of $r$, if $|\la|=\sum_i\la_i=r$ (resp., if, in addition, $\la_1\geq\la_2\geq\ldots$). For a composition
$\lambda$ of $r$, let $\fS_\la$ be the corresponding standard Young (parabolic) subgroup of $\fS_r$.
Let $\msD_\lambda$ be the set of minimal length right coset representatives of $\fS_\lambda$ in $\fS_r$. Then
$\msD_\lambda^{-1}=\{d^{-1}\;|\; d\in \msD_\lambda\}$ is the set of
minimal length left coset representatives of $\fS_\lambda$ in $\fS_r$.

Set
$x_\la=\sum_{w\in\fS_\la}T_w\in\Hr\in\cysHr.$
Then (see, e.g., \cite[Lem.~7.32]{DDPW})
\begin{equation} \label{commutator-x-T}
x_\la T_i=T_i x_\la=qx_\la,\;\forall\,i\in J_\la,
\end{equation}
where
\begin{equation}\label{Jla}
J_\la=\big\{1\leq i\leq r-1 \mid i\not=\sum_{j=1}^s \la_j; \forall\,1\leq s< n\big\}.
\end{equation}
We may use a similar condition to characterise the ``permutation modules''  $x_\la\cysHr$ and $\cysHr  x_{\la}$.
\begin{Lem}\label{permutation mod} Let $\lambda\in \Lambda(n,r)$. We have
$$\aligned
x_\la\cysHr &=\{h\in \cysHr \mid T_ih=qh,\;\forall i\in J_\la\},\\
\cysHr  x_{\la}&=\{h\in \cysHr \mid hT_i=qh,\;\forall i\in J_\la\}.\\
\endaligned$$
The two modules are $\scR$-free with respective bases $\{x_\la T_dL^\bfa \mid d\in\msD_\la ,\; \bfa\in\mbz_m^r\}$
and $\{L^\bfa T_d x_\la\mid d\in\msD_\la^{-1} ,\; \bfa\in\mbz_m^r\}.$ Moreover, we have an $\scR$-module
isomorphism
$$\Hom_{\cysHr}(x_\mu \cysHr,x_\la\cysHr)\cong x_\la\cysHr \cap \cysHr x_\mu.$$
\end{Lem}

\begin{proof}
By Lemma \ref{pbw basis}, $\{T_wL^{\bfa}\mid w\in\fS_r, \bfa\in\mbz_m^r\}$
is an $\scR$-basis of $\cysHr$. Suppose $h=\sum_{w,\bfa}c_{w,\bfa}T_wL^{\bfa}\in \cysHr$ and $T_ih=qh,\;\forall i\in J_\la$.
Equating gives $T_i\sum_{w}c_{w,\bfa}T_w=q\sum_{w}c_{w,\bfa}T_w$ for all $i\in J_\la, \bfa\in\mbz_m^r$.
Now our assertion follows from the fact that $x_\la\sH(r) =\{h\in \sH(r) \mid T_ih=qh,\;\forall i\in J_\la\}$
(see \cite{DJ86} or  \cite[Lem.~7.33]{DDPW}).
 The basis assertion follows from the fact that $\{x_\la T_d\mid d\in\msD_\la\}$
is an $\scR$-basis of $x_\la\sH(r)$. Finally, if $f\in\Hom_{\cysHr}(x_\mu \cysHr,x_\la\cysHr)$,
then $f(x_\mu)T_i=qf(x_\mu)$ for all $i\in J_\mu$. By the first assertion, $f(x_\mu)\in\cysHr x_\mu$
and so $f(x_\mu)=h_fx_\mu$ for some $h_f\in\cysHr$. Hence, the right module homomorphism $f$ is completely
defined by the left multiplication by $h_f$.
\end{proof}

We now consider the set of all compositions of $r$ into $n$ parts:
$$\Lambda(n,r)=\big\{\la=(\la_1, \ldots, \la_n)\in\mbn^n\mid\sum_{1\leq i\leq n}\la_i=r\big\}.$$
For $\bullet\in\{\bfu,\vtg\}$, let $\sT_\bullet(n,r)=\oplus_{\la\in\La(n,r)}x_\la \busHrR$. The
endomorphism algebra
$$\aligned
\sS_\bfu(n,r)&=\cysSrR=\End_{\cysHr}(\sT_\bfu(n,r))\\
 \text{resp., } \quad\sS_\vtg(n,r)&=\afsSrR=\End_{\afsHr}(\sT_\vtg(n,r))\endaligned
 $$
is called the {\it slim cyclotomic $q$-Schur algebra}, resp., {\it the affine $q$-Schur algebra}.

Note that $\sS_\bfu(n,r)$ is indeed a centraliser algebra
of the cyclotomic $q$-Schur algebra introduced in \cite{DJM}, which is called the $q$-Schur$^{2\textsc{b}}$
algebra in \cite{DS} when $\bfu=(-1,q_0)$. Note also that the algebra $\sS_\bfu(n,r)$ has been
studied in \cite{LR}. See also \cite{Gr97,DS} and more recently
in \cite[\S5.1]{BKLW} for the $m=2$ case.

For each $\mu\in \Lanr$, let $\frak l_\mu$ denote the following composition of maps
\begin{equation}\label{1la}
\frak l_\mu: \bigoplus_{\la\in\La(n,r)}x_\la \cysHr\stackrel{\pi}{\longrightarrow}x_\mu\cysHr \stackrel{\rm id}{\longrightarrow}x_\mu\cysHr\stackrel{\iota}{\longrightarrow}\bigoplus_{\la\in\La(n,r)}x_\la \cysHr
 \end{equation}
 where $\pi, \iota$ are the canonical projection and inclusion, respectively.
This is an idempotent $\frak l_\mu^2=\frak l_\mu$ which defines a {\it centraliser algebra} of
$\cysSr$:
\begin{equation}\label{mumu}
\frak l_\mu\sS_\bfu(n,r)\frak l_\mu\cong \End_{\cysHr}(x_\mu\cysHr).
\end{equation}

The next two sections are devoted to constructing explicitly an $\scR$-basis
of $\sS_\bfu(n,r)$.

\section{The commutative algebra $\sS_\bfu(1,r)$}

In this section we deal with the special composition $\la=(r,0,\ldots,0)\in\Lanr$. Our aim
is to construct an $\scR$-basis of $x_\la\cysHr\cap \cysHr x_\la$
and to prove that the centraliser algebra $\frak l_\la\sS_\bfu(n,r)\frak l_\la\cong \sS_\bfu(1,r)$
is commutative. For simplicity, we write
 $$x_{(r)}:=x_\la=\sum_{w\in\fS_r}T_w\;\text{ and }\; {\frak l}_{(r)}:={\frak l}_\la\in\sS_\bfu(n,r).$$

First of all, we have by \eqref{commutator-x-T} that
$$T_ix_{(r)}=x_{(r)}T_i=qx_{(r)},\;\forall\,1\leq i\leq r-1.$$
This together with Lemma \ref{pbw basis} implies the following.

\begin{Lem}\label{permutation mod r}
{\rm (1)}  The $\cysHr$-modules $x_{(r)}\cysHr $ and $\cysHr x_{(r)}$ are $\scR$-free with, respectively,  bases
 $\{x_{(r)} L^\bfa \;|\; \; \bfa\in\mbz_m^r\}$ and $\{L^\bfa x_{(r)}\;|\; \; \bfa\in\mbz_m^r\}.$

{\rm (2)} Likewise, the $\afsHr$-modules $x_{(r)}\afsHr $ and $\afsHr x_{(r)}$ are $\scR$-free with, respectively,  bases
 $\{x_{(r)} X^\bfa \;|\; \; \bfa\in\mbz^r\}$ and $\{X^\bfa x_{(r)}\;|\; \; \bfa\in\mbz^r\}.$
\end{Lem}

The following result provides a description of elements in $x_{(r)}\cysHr\cap \cysHr x_{(r)}$.

\begin{Lem}\label{sym}
Suppose that $$z=\sum_{\bfb\in \mbz_m^r}c_\bfb x_{(r)}L^\bfb\in x_{(r)}\cysHr\cap \cysHr x_{(r)},
\;\text{ where $c_\bfb\in \scR$.}$$
Then $c_\bfb=c_{\bfb w}$ for all $\bfb\in \mbz_m^r$ and $w\in \fS_r$.
\end{Lem}

\begin{proof} It suffices to prove that $c_\bfa=c_{\bfa s_i}$ for each fixed $\bfa\in \mbz_m^r$ and $1\leq i<r$.
Write $\bfa=(a_1,\ldots,a_r)$. If $a_i=a_{i+1}$, then
$\bfa=\bfa s_i$ and hence, $c_\bfa=c_{\bfa s_i}$, as desired. Now let $a_i\not=a_{i+1}$.
We may suppose $a_i<a_{i+1}$ (Otherwise, we replace $\bfa$ by $\bfa s_i$).

By Lemma \ref{commutator}, we obtain that
 $$\aligned
 z T_i=&\sum_{\bfb\in \mbz_m^r}c_\bfb x_{(r)}L^\bfb T_i\\
  =&\sum_{\bfb\in \mbz_m^r}c_\bfb x_{(r)}T_i L^{\bfb s_i} +\sum_{\bfb\in \mbz_m^r\atop b_i< b_{i+1}}(q-1)c_\bfb x_{(r)}\sum_{t=1}^{b_{i+1}-b_i}L^{\bfb s_i-t\alpha_i}\\
 &+\sum_{\bfb\in \mbz_m^r\atop b_i> b_{i+1}}(1-q)c_\bfb x_{(r)}\sum_{t=0}^{b_{i}-b_{i+1}-1}L^{\bfb s_i+t\alpha_i}.\\
\endaligned$$

Since $z\in \cysHr x_{(r)}$ and $x_{(r)}T_i=qx_{(r)}$, it follows that $qz=z T_i$ which
gives rise to the equality
$$\aligned
 \sum_{\bfb\in \mbz_m^r}q c_\bfb x_{(r)}L^\bfb=&\sum_{\bfb\in \mbz_m^r}q c_\bfb x_{(r)} L^{\bfb s_i}
+\sum_{\bfb\in \mbz_m^r\atop b_i< b_{i+1}}(q-1)c_\bfb x_{(r)}\sum_{t=1}^{b_{i+1}-b_i}L^{\bfb s_i-t\alpha_i}\\
 &\qquad\qquad\qquad \quad+\sum_{\bfb\in \mbz_m^r\atop b_i> b_{i+1}}(1-q)c_\bfb x_{(r)}\sum_{t=0}^{b_{i}-b_{i+1}-1}L^{\bfb s_i+t\alpha_i}.
\endaligned$$
 For the fixed $\bfa\in\mbz_m^r$, comparing the coefficients of $x_{(r)}L^\bfa$ on both sides
 implies that
\begin{equation} \label{coefficients-comp}
qc_\bfa=qc_{\bfa s_i}+c'+c'',
\end{equation}
 where $c'$ (resp., $c''$) denotes the coefficient of $x_{(r)}L^\bfa$ in the second
(resp., third) sum of the right hand side. In the following we calculate $c'$ and $c''$.

Consider those $\bfb=(b_1, \ldots, b_r)\in\mbz_m^r$ with $b_i< b_{i+1}$ such that
$$\bfb s_i-t\alpha_i=\bfa\;\text{ for some $1\leq t\leq b_{i+1}-b_i$.}$$
 With $\bfa$ fixed, it is equivalent to determine the precise range of $t$ for
which $\bfb=\bfa s_i-t\alpha_i$. Since $b_i=a_{i+1}-t$ and $b_{i+1}=a_i+t$, it
follows that $t\leq b_{i+1}-b_i=a_i-a_{i+1}+2t$. Hence,
$$a_{i+1}-a_i\leq t.$$
 On the other hand, since $0\leq  b_i=a_{i+1}-t\leq m-1$ and $0\leq b_{i+1}=a_i+t\leq m-1$,
it follows that $t\leq a_{i+1}$ and $t\leq m-1-a_i$. Consequently,
$$ a_{i+1}-a_i\leq t\leq {\rm min}\{a_{i+1}, m-1-a_i\}.$$
 Since $\bfa s_i-t\alpha_i=\bfa-(t-(a_{i+1}-a_i))\alpha_i$, by substituting
$t$ for $t-(a_{i+1}-a_i)$, we obtain that $\bfb=\bfa-t\alpha_i$
with
$$0\leq t\leq {\rm min}\{a_i, m-1-a_{i+1}\}.$$

Similarly, those $\bfb=(b_1, \ldots, b_r)\in\mbz_m^r$ with $b_i>b_{i+1}$ such that
$$\bfb s_i+t\alpha_i=\bfa\;\text{ for some $0\leq t\leq b_{i}-b_{i+1}-1$}$$
are simply $\bfb=\bfa s_i+t\alpha_i$ for all
$$0\leq t\leq {\rm min}\{a_i, m-1-a_{i+1}\}.$$

We treat the following two cases.

{\bf Case 1:} $a_{i}\leq m-1-a_{i+1}$. Then
$$c'=(q-1)\sum_{t=0}^{a_i}c_{\bfa-t\alpha_i}
\;\text{ and }\;
c''=(1-q)\sum_{t=0}^{a_{i}}c_{\bfa s_i+t\alpha_i}.$$
  Hence, we obtain from \eqref{coefficients-comp} that
$$\aligned
qc_\bfa&=qc_{\bfa s_i}+(q-1)\sum_{t=0}^{a_i}c_{\bfa-t\alpha_i}
+(1-q)\sum_{t=0}^{a_{i}}c_{\bfa s_i+t\alpha_i}\\
&=qc_{\bfa s_i}+(q-1)\big(c_\bfa+\sum_{t=1}^{a_i}c_{\bfa-t\alpha_i}\big)
+(1-q)\big(c_{\bfa s_i}+\sum_{t=1}^{a_{i}}c_{\bfa s_i+t\alpha_i}\big).
\endaligned$$
 This together with the fact that $\bfa s_i+t\alpha_i=(\bfa-t\alpha_i)s_i$ implies that
$$c_\bfa-c_{\bfa s_i}=\sum_{t=1}^{a_i}(q-1)\big(c_{\bfa-t\alpha_i}-c_{(\bfa-t\alpha_i)s_i}\big).$$
If $a_i=0$, then $c_\bfa=c_{\bfa s_i}$. Let $1\leq k\leq m-1$ and suppose $c_\bfa=c_{\bfa s_i}$
in case $a_i< k$. Now consider the case $a_i=k$. Since for each $1\leq t\leq a_i$,
$\bfa-t\alpha_i=:{\bf b}^{(t)}=(b_1^{(t)},\ldots,b^{(t)}_r)$ satisfies $b^{(t)}_i=a_i-t<a_i=k$,
$$b^{(t)}_i=a_i-t<a_{i+1}+t=b^{(t)}_{i+1}\;\text{ and }\;
b^{(t)}_i=a_i-t\leq m-1-a_{i+1}-t=m-1-b^{(t)}_{i+1},$$
 we have $c_{\bfa-t\alpha_i}=c_{(\bfa-t\alpha_i)s_i}$, $\forall\,1\leq t\leq a_i$.
Hence, $c_\bfa=c_{\bfa s_i}$.

\medskip

{\bf{Case 2:}} $a_{i}> m-1-a_{i+1}$. In this case,
$$c'=(q-1)\sum_{t=0}^{m-1-a_{i+1}}c_{\bfa-t\alpha_i}\;
\text{ and }\;
c''=(1-q)\sum_{t=0}^{m-1-a_{i+1}}c_{\bfa s_i+t\alpha_i}.$$
 From \eqref{coefficients-comp} it follows that
$$qc_\bfa=qc_{\bfa s_i}+(q-1)\sum_{t=0}^{m-1-a_{i+1}}c_{\bfa-t\alpha_i}
+(1-q)\sum_{t=0}^{m-1-a_{i+1}}c_{\bfa s_i+t\alpha_i}.$$
 Using an argument similar to that in Case 1, we obtain that
$$c_\bfa-c_{\bfa s_i}=\sum_{t=1}^{m-1-a_{i+1}}(q-1)\big(c_{\bfa-t\alpha_i}-c_{(\bfa-t\alpha_i)s_i}\big).$$
 Finally, proceeding by induction on $m-1-a_{i+1}$ gives the equality $c_\bfa=c_{\bfa s_i}$.
\end{proof}

\begin{Rem}\label{aff1} In Lemma \ref{afsym} in the appendix, we will modify the proof to get the affine version of the lemma.
\end{Rem}

Let $\scR[L_1, \ldots, L_r]_m$ be the free $\scR$-submodule of $\cysHr$
with basis $\{L^\bfa\mid \bfa\in\mbz_m^r\}$. Then the symmetric group $\fS_r$
acts on $\scR[L_1, \ldots, L_r]_m$ defined by
$$w\cdot L^\bfa=L^{\bfa w^{-1}},\;\forall\;w\in\fS_r,\,\bfa\in\mbz_m^r.$$
 By $\scR[L_1, \ldots, L_r]_m^{\fS_r}$ we denote the set of fixed
elements in $\scR[L_1, \ldots, L_r]_m$.

Consider the elementary symmetric polynomials
$\sigma_1, \ldots, \sigma_r$ in $L_1, \ldots, L_r$, i.~e.,
$$\sigma_i=\sigma_i(L_1,\ldots,L_r)=\sum_{1\leq j_1<\cdots <j_i\leq r}L_{j_1}\cdots L_{j_{i}},\;\forall\,1\leq i\leq r.$$
 It is clear that $\sigma_1,\ldots,\sigma_r\in \scR[L_1, \ldots, L_r]_m^{\fS_r}$.
Set
\begin{equation}\label{Gamma-m-r}
\La(r,<\!m)=\bigcup_{k=1}^{m-1}\La(r,k)=\big\{\bfa=(a_1, \ldots, a_r)\in\mbn^r \mid
\sum_{i=1}^ra_i\leq m-1\big\}.
\end{equation}
Then for each $\bfa=(a_1, \ldots, a_r)\in\La(r,<\!m)$, the monomial
$$\sigma^{\bfa}=\sigma^{\bfa}(L_1,\ldots,L_r)=\sigma_1^{a_1}\cdots\sigma_r^{a_r}$$
lies in $\scR[L_1, \ldots, L_r]_m^{\fS_r}$. Indeed, we have the following result.

\begin{Lem} \label{symm-polynomials} The set of fixed elements $\scR[L_1, \ldots, L_r]_m^{\fS_r}$
is a free $\scR$-module with basis
$$\{\sigma^\bfa\mid \bfa\in \La(r,<\!m)\}.$$
 Moreover, for each $1\leq i\leq r-1$ and $f\in\scR[L_1, \ldots, L_r]_m^{\fS_r}$,
$$T_if=fT_i.$$
\end{Lem}

\begin{proof} By a standard method on symmetric polynomials (see \cite[I, \S2.13]{Jac}), each
$f\in\scR[L_1, \ldots, L_r]_m^{\fS_r}$ can be written as
an $\scR$-linear combination of monomials
$$\sigma_1^{t_1-t_2}\sigma_2^{t_2-t_3}\cdots \sigma_{r-1}^{t_{r-1}-t_r}\sigma_r^{t_r},$$
 where $t_1,\ldots,t_r\in\mbn$ satisfy
$$m-1\geq t_1\geq t_2\geq \cdots\geq t_r\geq 0.$$
 Therefore, $\scR[L_1, \ldots, L_r]_m^{\fS_r}$ is spanned by the set
$\{\sigma^\bfa\mid \bfa\in \La(r,<\!m)\}$.

For each $\bfa\in\La(r,<\!m)$, define
$\widehat{\bfa}=(\widehat{a}_1, \ldots, \widehat{a}_r)\in\mbz_m^r$
by setting
$$\widehat{a}_i=\sum_{i\leq j\leq r}a_j,\;\forall\,1\leq i\leq r.$$
Then the leading term of $\sigma^\bfa$ as a polynomial in $L_1,\ldots,L_r$
is $L^{\widehat{\bfa}}$ and, moreover, for $\bfa,\bfb\in \La(r,<\!m)$,
$$\bfa=\bfb \Longleftrightarrow \widehat\bfa=\widehat\bfb.$$
 Consequently, the linear independence of the set $\{\sigma^\bfa\mid \bfa\in \La(r,<\!m)\}$
 follows from that of $\{L^\bfa\mid\bfa\in\mbz_m^r\}$.

Let $1\leq i,j\leq r$ with $i\not=r$. By summing up both sides of the formula \eqref{Bernstein-formula}
over $L^\bfa=L_{k_1}\ldots L_{k_j}$ with $1\leq k_1<\cdots<k_j\leq r$, we obtain that $T_i\sigma_j=\sigma_j T_i$.
Thus, the second assertion follows.
\end{proof}

\begin{Prop}\label{basis 1}
As an $\scR$-module, $x_{(r)}\cysHr \cap \cysHr x_{(r)}$ is free with basis
$$\mathcal B_{r,r}=\mathcal{B}_{r,r}^{(m)}=\{x_{(r)}\sigma^\bfa=\sigma^\bfa x_{(r)}\;|\; \bfa\in\La(r,<\!m)\}.$$
\end{Prop}

\begin{proof} By Lemma \ref{sym}, each element in $x_{(r)}\cysHr \cap \cysHr x_{(r)}$
has the form $x_{(r)}f$ for some $f\in\scR[L_1, \ldots, L_r]_m^{\fS_r}$. This together
with Lemma \ref{symm-polynomials} implies that $x_{(r)}\cysHr \cap \cysHr x_{(r)}$
is spanned by the set $\sB_{r,r}$. Finally, the linear independence of $\sB_{r,r}$ follows from Lemma
\ref{symm-polynomials} and the fact that $\{T_wL^\bfa\mid w\in\fS_r,\bfa\in\mbz_m^r\}$
is an $\scR$-basis of $\cysHr$.
\end{proof}

Combining the results above gives the following corollary.

\begin{Coro}\label{old2.3} The centraliser algebra $\sS_\bfu(1,r)=\frak l_{(r)}\sS_\bfu(n,r)\frak l_{(r)}$ is a commutative
$\scR$-algebra.
\end{Coro}

\begin{proof} By \eqref{mumu}, there is an algebra isomorphism
$$\frak l_{(r)}\sS_\bfu(n,r)\frak l_{(r)}\cong \End_{\cysHr}(x_{(r)}\cysHr)=\sS_\bfu(1,r).$$
The latter as an $\scR$-module is isomorphic to $x_{(r)}\cysHr\cap\cysHr x_{(r)}$, by Lemma \ref{permutation mod}.
For each $\bfa\in\La(r,<\!m)$, if we define a right $\cysHr$-module homomorphism
$$\Phi_\bfa: \;\; x_{(r)}\cysHr\longrightarrow x_{(r)}\cysHr, \;\;\;x_{(r)}h\longmapsto \sigma^\bfa x_{(r)}h,\;\forall h\in\cysHr,$$
then the proposition above implies that $\{\Phi_\bfa\; |\; \bfa\in \La(r,<\!m)\}$
forms an $\scR$-basis for $ \End_{\cysHr}(x_{(r)}\cysHr)$. Finally, it follows from Lemma \ref{symm-polynomials} that $\Phi_\bfa\circ\Phi_\bfb=\Phi_\bfb\circ\Phi_\bfa$ for all $\bfa,\bfb\in \La(r,<\!m)$.
Hence, $\sS_\bfu(1,r)$ is commutative, as desired.
\end{proof}

For $t\in\mbn$ and $\nu=(\nu_1, \ldots, \nu_t)\in\Lambda(t, r)$, let
 \begin{equation}
 \aligned
 \La(\nu,<\!m)&=\La(\nu_1,<\!m)\times\cdots\times\La(\nu_t,<\!m)\\
 &=\{(\bfa^{(1)}, \ldots, \bfa^{(t)})\mid \bfa^{(i)}\in\La(\nu_i,<\!m),\;\forall\,1\leq i\leq t\},
\endaligned
\end{equation}
 where $\La(\nu_i,<m)=\{{\bf b}\in\mbn^{\nu_i}\mid |{\bf b}|<m\}$ as defined
 in \eqref{Gamma-m-r}.

\begin{Def}\label{X-nu}
 For each $\nu=(\nu_1, \ldots, \nu_t)\in\Lambda(t, r)$ and each $1\leq i\leq t$, let $\sigma_1^{(i)}, \ldots, \sigma_{\nu_i}^{(i)}$
denote the elementary symmetric polynomials in
$$L_{\nu_1+\cdots+\nu_{i-1}+1}, \ldots, L_{\nu_1+\cdots+\nu_{i-1}+\nu_i}.$$
For every $\bfa(\nu)=(\bfa^{(1)}, \ldots, \bfa^{(t)})\in \La(\nu,<\!m)$ with $\bfa^{(i)}=(a_1^{(i)}, \ldots, a_{\nu_i}^{(i)})$, define
$$\sigma^{\bfa(\nu)}:=(\sigma_1^{(1)})^{a_1^{(1)}}(\sigma_2^{(1)})^{a_2^{(1)}}\cdots (\sigma_{\nu_1}^{(1)})^{a_{\nu_1}^{(1)}}\cdots\cdots
(\sigma_1^{(t)})^{a_1^{(t)}}(\sigma_2^{(t)})^{a_2^{(t)}}\cdots (\sigma_{\nu_t}^{(t)})^{a_{\nu_t}^{(t)}}.
$$
\end{Def}

For example, let $m=3$, $t=4$ and $r=7$, and take $\nu=(1,2,1,3)$.
Then, for $\bfa(\nu)=((1), (1,1), (1), (1,0,1))$, $\bfb(\nu)=((1), (1,0), (0), (0,2,0))\in \La(\nu,<3)$, we have
$$\sigma^{\bfa(\nu)}=L_1\cdot\big((L_2+L_3)L_2L_3\big)\cdot L_4\cdot \big((L_5+L_6+L_7)L_5L_6L_7\big)$$
and
$$ \sigma^{\bfb(\nu)}=L_1\cdot (L_2+L_3)\cdot(L_5L_6+L_5L_7+L_6L_7)^2.$$

\vspace{.3cm}

Let $\nu=(\nu_1,\ldots,\nu_t)\in \Lambda(t,r)$ for some $t\in\mbn$. The action
of $\fS_r$ on $\scR[L_1, \ldots, L_r]_m$ restricts to an action of its subgroup $\fS_{\nu}$.
The following result will be needed in the next section. Recall the set $J_\la$ defined in \eqref{Jla}.

\begin{Prop} \label{symm-funct-composition}
{\rm (1)} The set of fixed elements $\scR[L_1, \ldots, L_r]_m^{\fS_{\nu}}$ is a free $\scR$-module
with basis
$\{ \sigma^{\bfa(\nu)}\mid \bfa(\nu)\in\La(\nu,<\!m)\}.$
 Moreover, for $i\in J_\nu$ and $f\in \scR[L_1, \ldots, L_r]_m^{\fS_{\nu}}$, we have
$T_if=fT_i$ in $\cysHr$.

{\rm (2)} Suppose that $z\in x_{\nu}\cysHr \cap \cysHr  x_{\nu}$ has the ``pure''  form
$$z=\sum_{\bfa\in\mbz_m^r}g_\bfa x_{\nu}L^\bfa=\sum_{\bfb\in\mbz_m^r}h_\bfb L^\bfb x_{\nu},\quad\text{where $g_\bfa$, $h_\bfb\in\scR$.}$$
 Then $z$ is an $\scR$-linear combination of
$x_{\nu}\sigma^{\bfa(\nu)}$ with $\bfa(\nu)\in \La(\nu,<\!m)$.
\end{Prop}

\begin{proof} (1) This statement follows from an argument similar to that in the proof
of Proposition \ref{basis 1} together with an induction on $t$.

(2) Clearly, for each $i\in J_\nu$ (i.~e., $1\leq i\leq r-1$ with $i\not=\sum_{j=1}^s \nu_j$,
$\forall\,1\leq s\leq t-1$), we have
$$T_ix_\nu=x_\nu T_i=qx_\nu \;\text{ and }\; zT_i=qz.$$
 By applying similar arguments to those in the proof of Lemma \ref{sym}, we obtain that
$$g_\bfa=g_{\bfa s_i},\;\forall\,i\in J_\nu.$$
 Thus, $z=x_\nu f$ with $f\in \scR[L_1, \ldots, L_r]_m^{\fS_{\nu}}$. The statement then
 follows from (1).
\end{proof}

\begin{Rem}\label{aff2} By Lemma \ref{afsym} and a similar argument,
the $\scR$-module $x_{(r)}\afsHr \cap \afsHr x_{(r)}$ contains the linearly independent set
$\{x_{(r)}\sigma^\bfa(X_1,\cdots,X_r)\mid\bfa\in\mbn^r\}$.
Hence, there is an affine version of Proposition \ref{symm-funct-composition}.
\end{Rem}

\section{A new integral basis for $\sS_\bfu(n,r)$}

In this section we describe an $\scR$-basis for $x_{\la}\cysHr \cap \cysHr x_{\mu}$ ($\la, \mu\in \Lambda(n,r)$), all of which give rise to a new basis for $\sS_\bfu(n,r)$.
We keep all the notations in previous sections.

For two compositions
$\la$ and $\mu$ of $r$, let $\msD_{\lambda,\mu}=\msD_\lambda\cap \msD_\mu^{-1}$.
Then $\msD_{\lambda,\mu}$ is the set of
minimal length $\fS_\lambda$-$\fS_\mu$ double coset representatives in $\fS_r$.

\begin{Lem}{\rm(}\cite{Ca}, \cite[Lem. 1.6]{DJ86}{\rm)}\label{Basic}
Let $\lambda, \mu$ be two compositions of $r$ and $d\in \msD_{\lambda,\mu}$.
\begin{itemize}
\item[(1)] There exist compositions $\nu(d)=\la d\cap\mu, \nu(d^{-1})=\la\cap\mu d^{-1}$ of $r$
such that $$\fS_{\nu(d)}=d^{-1}\fS_\lambda d\cap \fS_\mu,\quad \fS_{\nu(d^{-1})}=\fS_\lambda \cap d \fS_\mu d^{-1}.$$

\item[(2)] Each element $w\in \fS_\lambda d\fS_\mu$ can be uniquely written
as $w=udv=u'dv'$ with $u\in\fS_\lambda$, $v\in \msD_{\nu(d)}\cap \fS_\mu$, $u'\in \msD^{-1}_{\nu(d^{-1})}\cap \fS_{\la}$ and $v'\in \fS_\mu$,
which satisfy $\ell(w)=\ell(u)+\ell(d)+\ell(v)=\ell(u')+\ell(d)+\ell(v')$. (Here $\ell$ is the length function.)
\end{itemize}
In particular, we have
\begin{equation}\label{element}
x_\la T_{d}=\big(\sum_{u\in \msD^{-1}_{\nu(d^{-1})}\cap \fS_{\la}}T_u\big)T_{d}x_{\nu(d)}\;\text{ and }\; T_dx_{\mu}=T_{d}x_{\nu(d)}\big(\sum_{v\in \msD_{\nu(d)}\cap \fS_\mu}T_v\big).
\end{equation}
\end{Lem}

For $\la, \mu\in \Lambda(n,r), \; d\in \msD_{\la, \,\mu}$, the double coset $\fS_\lambda d\fS_\mu$ defines an $n\times n$ matrix
\begin{equation}\label{theta}
\theta(\la,d,\mu)=\big(|R_i^{\la}\cap d(R_j^{\mu})|\big),\end{equation}
where, for $\nu=(\nu_1,\nu_2,\ldots,\nu_n)\in\La(n,r)$,
\begin{equation}\label{partial sum}
R_i^\nu=\{\widetilde\nu_{i-1}+1,\; \widetilde\nu_{i-1}+2,\ldots, \widetilde\nu_{i-1}+\nu_i\}\text{ with }\widetilde\nu_0=0,
\,\widetilde \nu_{k}=\sum_{j=1}^{k}\nu_j,
\end{equation} see, e.g., \cite[Thm 4.15]{DDPW}. Let
$$\Theta(n,r)=\big\{\theta(\la,d,\mu)\mid \la, \mu\in \Lambda(n,r), \; d\in \msD_{\la, \,\mu}\}\;\text{ and }\;
\Theta(n)=\bigcup_{r\geq0}\Theta(n,r).$$
Then $\Theta(n)$ is the set $M_n(\mbn)$ of $n\times n$ matrices with nonnegative integer coefficients.
For each $A=(a_{ij})\in\Theta(n)$, define its row and column vectors by
$$\ro(A)=\big(\sum_{j=1}^n a_{1j}, \ldots,\sum_{j=1}^n a_{nj}\big),\;
\co(A)=\big(\sum_{i=1}^n a_{i1}, \ldots, \sum_{i=1}^n a_{in}\big)\in\mbn^n.$$
Thus, if $A=\theta(\la,d,\mu)$ and $\nu(d)$ as defined in Lemma \ref{Basic}(1), then
$\la=\ro(A)$, $\mu=\co(A)$, and
\begin{equation}\label{nu(A)}
\nu(d)=\nu(A)=:(a_{11}, \ldots, a_{n1},a_{12}, \ldots, a_{n2}, \ldots, a_{1n}, \ldots, a_{nn})\in\La(n^2,r).
\end{equation}
We also write $d=d_A$. Hence, the inverse map of $\theta$  has the form
\begin{equation}\label{bij}\aligned
\theta^{-1}: \Theta(n,r)&\longrightarrow \{(\la, d, \mu)\;|\; \la, \mu\in \Lambda(n,r), \; d\in \msD_{\la, \,\mu}\} ,\\
A&\longmapsto (\row(A),\, d_A, \,\col(A)),
\endaligned
\end{equation}
and the subset $\Theta(n,r)_{\la,\mu}=\{A\in\Theta(n,r)\mid \ro(A)=\la,\co(A)=\mu\}$ identifies the double cosets in $\fS_\la\backslash\fS_{r}/\fS_\mu$.

C. Mak \cite[Thm 4.1.2]{Mak} has generalised this double coset correspondence from $\fS_r$ to $\fS_{m,r}$ ($m\geq 1)$.

The group $\fS_{m,r}=\mbz_m\wr\fS_r$ admits a presentation with generators $s_0,s_1,\ldots,s_{r-1}$ and relations:
\begin{equation}\label{presentn}\aligned
{}& s_0^m=1,\; s_i^2=1\; (1\leq i\leq r-1),\;s_0s_1s_0s_1=s_1s_0s_1s_0,\\
& s_is_{i+1}s_i=s_{i+1}s_is_{i+1}\; (1\leq i< r-1),\;s_is_j=s_js_i\;(|i-j|>1).
\endaligned
\end{equation}

 There is an alternative description of $\fS_{m,r}$, following \cite{Mak}.
Let $\xi$ be a primitive $m$-th root of unity. Set $C_m=\{1, \xi, \ldots, \xi^{m-1}\}$
and $R=\{1,2,\ldots, r\}$. Consider the set
$$C_mR=\{\xi^a i\mid 0\leq a\leq m-1, \; 1\leq i\leq r\}.$$
 Then $\fS_{m,r}$ is identified with the group of all permutations $w$ of $C_mR$ satisfying
$$w(\xi^a i)=\xi^a w(i),\;\forall\, 1\leq i\leq r.$$
  In other words, each element $w$ in $\fS_{m,r}$ is uniquely determined by the sequence
$(w(1), \ldots, w(r))$. In particular, $s_0$ corresponds to the sequence $(\xi, 2,\ldots, r)$
and for $1\leq i\leq r-1$, $s_i$ corresponds to the sequence $(1, 2,\ldots,i-1, i+1, i, i+2,\ldots, r)$

Let $\cnrm$ be the set of $n\times nm$-matrices with nonnegative integer entries sum to $r$.
We will write elements in $\cnrm$ as an $n\times n$ array with $m$-tuples of nonnegative integers as entries. More precisely,
\begin{equation}\label{da}
\cnrm =\big\{\da=(\mathbf{a}_{ij})_{n\times n} \mid \mathbf{a}_{ij}\in\mbn^m,
     1\leq i, j\leq n; \sum_{i,j=1}^n|\mathbf{a}_{ij}|=r\big\}.
\end{equation}
 Then each element $\da\in \cnrm$ gives a matrix
$|\da|:=(|\bfa_{ij}|)_{n\times n}\in \Theta(n,r).$ Thus, we obtain a matrix-valued map
\begin{equation*}
\mathfrak{f}=|\;\;|:\cnrm \lra \Theta(n,r),\quad (\mathbf{a}_{ij})\longmapsto (|\mathbf{a}_{ij}|).
\end{equation*}
For $\la, \mu\in \Lambda(n,r)$, let
$$\cnrm_{\la, \mu}=\{ \da\in \cnrm\mid \row(|\da|)=\la, \; \col(|\da|)=\mu\}.$$
The following can be seen from {\cite[Thm 4.1.2]{Mak}} and  \eqref{bij}.

\begin{Lem}\label{from da to a} Let $\la,\mu\in\La(n,r)$.
\begin{itemize}
\item[(1)] If $A=\th(\la,d,\mu)$ with $d\in \dlm$, then the fibre of $\mathfrak f$ at $A$ is
 $$\mathfrak{f}^{-1}(A)=\cnrm_{\la,\, \mu}^d:=\{\da\in \cnrm_{\la,\mu}\mid d_{|\da|}=d\}.$$

\item[(2)] There is a bijection  $\theta_{\la,\mu}=\theta_{\la, \mu}^{(m)}:\;
 \fS_\la\backslash\fS_{m,r}/ \fS_\mu\rightarrow \cnrm_{\la, \mu}$ sending a double coset
$\fS_\la w\fS_\mu$ to $ \da=(\bfa_{ij})_{n\times n},$
 where $\bfa_{ij}=(a_{ij}^{(1)}, \ldots, a_{ij}^{(m)})$ is defined by $a_{ij}^{(t)}=|\xi^t R_i^{\la}\cap w(R_j^{\mu})|$ for $1\leq t\leq m$.
\end{itemize}
\end{Lem}

For each $\da\in \cnrm_{\la, \mu}$, define a composition,  as in \eqref{nu(A)},  associated with the parabolic subgroup $\fS_{\ro(|\da|)}^{d_{|\da|}}\cap\fS_{\co(|\da|)}$:
\begin{equation}\label{nu|da|}
\nu(|\da|)=\nu(d_{|\da|})=(|\bfa_{11}|,\ldots, |\bfa_{n1}|,|\bfa_{12}|,\ldots,|\bfa_{n2}|,\ldots,
|\bfa_{1n}|,\ldots, |\bfa_{nn}|)\in\Lambda(n^2,r),
\end{equation}
We also define a composition
\begin{equation}\label{nuda}
\nu(\da)=(\bfa_{11}, \ldots, \bfa_{n1}, \bfa_{12}, \ldots, \bfa_{n2}, \ldots, \bfa_{1n}, \ldots, \bfa_{nn})
\in\La(mn^2,r).\end{equation}
 Then we have the associated Young subgroup $\fS_{\nu(\da)}$ of $\fS_r$.
The following lemma generalizes \cite[Lem.~1.6]{DJ86} (for type $A$) and
can be deduced from \cite[Lem.~4.3.1(d)]{Mak}. However, we provide a proof here for completeness and later use.

 Let
\begin{equation}\label{sIsJ}
\sI=[1,n]^2,\quad \sJ=\sJ(\da)=\{(i,j)\in\sI\mid 1\leq i,j\leq n, a_{ij}^{(m)}<|\bfa_{ij}|\}.
\end{equation}
 We order $\sI$ by setting
 \begin{equation}\label{prec}
 (i,j)\preceq(i',j')\text{ if }j<j' \text{ or }i\leq i'\text{ whenever }j=j'.
 \end{equation}

For $1\leq i,j\leq n$, $1\leq k\leq m$, let
\begin{equation}\label{aijk}
a(i,j,k)=\widetilde a_{ij}+\sum_{1\leq x\leq k-1}a_{ij}^{(x)},\;\;\text{ where }\widetilde a_{ij}=\sum_{1\leq i'\leq n, \;l<j}|\bfa_{i'l}|+\sum_{1\leq i'\leq i-1}|\bfa_{i'j}|
\end{equation}
 is the partial sum of the sequence in \eqref{nu|da|} up to $|\bfa_{ij}|$, but not include $|\bfa_{ij}|$.

For any $(i,j)\in\sI$ and $1\leq k\leq m$, let
\begin{equation}\label{tijk319}
\aligned
t_1&=s_0, t_i=s_{i-1}\cdots s_2s_1t_1s_1s_2\cdots s_{i-1}, \quad2\leq i\leq r;\\
t_{ij}^{(k)}&=\begin{cases}
t_{a+1}t_{a+2}\cdots t_{a+a_{ij}^{(k)}},&\text{ if }(i,j)\in \sJ, k\leq m-1, a_{ij}^{(k)}\geq1;\\
1,&\text{ otherwise,}\end{cases}
\endaligned\end{equation}
where $a=a(i,j,k)$.
\begin{Lem}\label{7 crucial} For $\la, \mu\in \Lambda(n,r)$ and $w\in \fS_{m,r}$, if $\da=({\bf a}_{ij})=\theta_{\la, \mu}(\fS_\la w\fS_\mu)\in\cnrm_{\la, \mu}$,
then $w^{-1}\fS_\la w\cap \fS_\mu=\fS_{\nu(\da)}.$
Moreover, we may choose
\begin{equation}\label{w=dt}
w=d_{|\da|}\cdot\prod_{(i,j)\in \mathcal{J}}\big((t_{ij}^{(1)})^{1}(t_{ij}^{(2)})^{2}\cdots (t_{ij}^{(m-1)})^{m-1}\big)
\end{equation}
(noting $(t_{ij}^{(m)})^{m}=1$).
\end{Lem}

\begin{proof} Write $\la=(\la_1, \ldots, \la_n)$ and $\mu=(\mu_1, \ldots, \mu_n)$.
For $1\leq i\leq n$, let $\fS_\la^{(i)}$ be the subgroup of $\fS_\la$ generated by
$s_{\widetilde\la_{i-1}+1}, \ldots, s_{\widetilde\la_{i-1}+\la_i-1}$ ($\widetilde\la_{i-1}=\sum_{1\leq j\leq i-1}\la_j$
and $\la_0=0$ by convention). Similarly, we obtain the subgroups $\fS_\mu^{(i)}$ of $\fS_\mu$.
Then
$$\fS_\la=\fS_\la^{(1)}\times \cdots \times\fS_\la^{(n)}\;\text{ and }\;
\fS_\mu=\fS_\mu^{(1)}\times \cdots \times\fS_\mu^{(n)}.$$
This implies that
$w^{-1}\fS_\la w=(w^{-1}\fS_\la^{(1)}w) \times \cdots \times (w^{-1}\fS_\la^{(n)}w).$
Hence,
\begin{equation}\label{sym form}
\aligned
w^{-1}\fS_\la w\cap \fS_\mu&=(w^{-1}\fS_\la w\cap \fS_\mu^{(1)})\times\cdots\times (w^{-1}\fS_\la w\cap \fS_\mu^{(n)})\\
&=\big(\prod_{i=1}^n w^{-1}\fS_\la^{(i)} w\cap \fS_\mu^{(1)}\big) \times \cdots\times
\big(\prod_{i=1}^n w^{-1}\fS_\la^{(i)} w\cap \fS_\mu^{(n)}\big).
\endaligned
\end{equation}

By the definition, $\fS_\la^{(i)}$ is the subgroup of $\fS_{m,r}$ consisting of
all the elements which fix the set $C_m(R\setminus R_i^{\la})$, and map $R_i^{\la}$ onto $R_i^{\la}$
(Thus, they map $\xi^tR_i^{\la}$ onto $\xi^tR_i^{\la}$ for all $1\leq t\leq m$).
Consequently, $w^{-1}\fS_\la^{(i)}w$ is a subgroup of $\fS_{m,r}$ consisting of
all the elements which fix the set $C_mw^{-1}(R\backslash R_i^{\la})$ and map $w^{-1}(R_i^{\la})$
onto $w^{-1}(R_i^{\la})$ (They also map $\xi^tw^{-1}(R_i^{\la})$ onto $\xi^tw^{-1}(R_i^{\la})$ for $1\leq t\leq m$).

Therefore, for $1\leq i, j\leq n$, we can write
 $$w^{-1}\fS_\la^{(i)} w\cap \fS_\mu^{(j)}=\prod_{t=1}^{m}\fS_{ij}^{(t)},$$
 where $\fS_{ij}^{(t)}$ is the subgroup formed by all elements in $\fS_{m,r}$ which
permute elements in $\xi^tw^{-1}(R_i^{\la})\cap  R_j^{\mu}$ and fix other elements in
$C_mw^{-1}(R_i^{\la})\cap  R_j^{\mu}$.

 By Lemma \ref{from da to a}(2), for $1\leq i, j\leq n$ and
 $1\leq k\leq m$,
 $$|\xi^kw^{-1}(R_i^{\la})\cap  R_j^{\mu}|=|\xi^kR_i^{\la}\cap w( R_j^{\mu})|=a_{ij}^{(k)}.$$
 Therefore, $\fS_{ij}^{(k)}$ is the subgroup of $\fS_{m,r}$
generated by $s_{a+1}, s_{a+2},\ldots, s_{a+a_{ij}^{(k)}-1}$,
where $a=a(i,j,k)$ as defined in \eqref{aijk}.
 This together with \eqref{sym form} implies that $w^{-1}\fS_\la w\cap \fS_\mu=\fS_{\nu(\da)}$.

 It remains to prove the last assertion. Assume now $w$ is given as in \eqref{w=dt}. For each element $u\in\fS_{m,r}$, we write
$u=(u(1), \ldots, u(r))$. Then, for $(i,j)\in\sJ$,$1\leq k\leq m$ and $a=a(i,j,k)$,
$$(t_{ij}^{(k)})^k=(1,2, \ldots, a, \xi^k(a+1), \xi^k(a+2), \ldots, \xi^k(a+a_{ij}^{(k)}), a+a_{ij}^{(k)}+1, a+a_{ij}^{(k)}+2, \ldots, r).$$
If we put $\beta_{ij}^{(k)}=(\xi^k(a+1), \xi^k(a+2), \ldots, \xi^k(a+a_{ij}^{(k)}))$,
then (after concatenation)
$$d_{|\da|}^{-1}w:=y=(\beta_{11}^{(1)}, \ldots,  \beta_{1,1}^{(m)}, \beta_{21}^{(1)}, \ldots,\beta_{2,1}^{(m)}, \ldots \beta_{n-1,n}^{(1)}, \ldots, \beta_{n-1,n}^{(m)},\beta_{n,n}^{(1)}, \ldots, \beta_{nn}^{(m)}).$$
Since $|\bfa_{ij}|=|d_{|\da|}^{-1}R_i^\la\cap  R_j^\mu|$, it follows that
 $$a_{ij}^{(k)}=|\xi^k(d_{|\da|}^{-1}R_i^{\la})\cap y(R_j^{\mu})|=|\xi^kR_i^{\la}\cap wR_j^{\mu}|,$$ for all $1\leq i,j\leq n$
 and $1\leq k\leq m$. Hence, $\th_{\la\mu}(\fS_\la w\fS_\mu)=\da$.
\end{proof}

We need more notations. For $k\in\mbn$ and $\lambda=(\lambda_1, \ldots, \lambda_m)\in \Lambda(m,k)$, let
$$b_i=\sharp\{t\in\{1,\ldots, m-1\}\mid \lambda_1+\cdots+\lambda_t\geq i\}\quad(i=1,2,\ldots,k).$$
 Then
$m-1\geq b_1\geq b_2\geq \cdots\geq b_k.$
 In other words, with the notation in \eqref{partial sum}, $(b_1,\ldots,b_k)$ is the partition dual to
$(\widetilde\la_{m-1},\ldots,\widetilde\la_2,\widetilde\la_1).$
 Finally, define
\begin{equation}\label{dd}
 \ddot{\la}=(b_1-b_2, \ldots, b_{k-1}-b_k, b_k),
\end{equation}
which clearly lies in $\La(k,<\!m)$. Indeed, it is easy to check that the correspondence
$\la\mapsto\ddot{\la}$ induces a bijection
\begin{equation} \label{bijection-lamda-gamma}
g_{m,k}: \Lambda(m,k)\longrightarrow \La(k,<\!m).
\end{equation}

For example, let $m=3$, $k=6$ and take $\la=(2, 3, 1)\in \Lambda(3, 6)$. Then
$$(b_1,b_2,b_3,b_4,b_5,b_6)=(2, 2, 1, 1, 1, 0)\;\text{ and }\;\ddot{\la}=(0, 1, 0, 0, 1, 0).$$

Let $\da=(\mathbf{a}_{ij})_{n\times n}\in \cnrm$ with $\bfa_{ij}=(a_{ij}^{(1)}, \ldots, a_{ij}^{(m)})$
for $1\leq i,j\leq n$. By the construction above, each $\bfa_{ij}$ as a composition
in $\Lambda(m,|\bfa_{ij}|)$ gives rise to a vector $\ddot{\bfa}_{ij}\in\La(|\bfa_{ij}|,<\!m)$ and a ``matrix'' $\ddot\da=(\ddot\bfa_{ij})$.
Thus, by juxtaposition, $\da$ defines a composition
$$\bfa(\ddot\da):=(\ddot{\bfa}_{11},\ldots, \ddot{\bfa}_{n1},\ddot{\bfa}_{12},\ldots,\ddot{\bfa}_{n2},\ldots,
\ddot{\bfa}_{1n},\ldots, \ddot{\bfa}_{nn})\in\La(\nu(|\da|),<\!m),$$
 which\footnote{Unlike the compositions $\nu(|\da|), \nu(\da)$ defined in \eqref{nu|da|},\eqref{nuda}, $\bfa(\ddot\da)$ is not a composition of $r$.} further defines, in the notation of Definition \ref{X-nu}, an element
\begin{equation}\label{sigddA}
\sigma^{\ddot\da}:=\sigma^{\bfa(\ddot\da)}\in \scR[L_1,\ldots,L_r]_m^{\fS_{\nu(|\da|)}}.
\end{equation}

\begin{Lem}\label{xim} Suppose $\la, \mu\in \Lambda(n,r)$, $d\in \msD_{\la, \mu}$ and let $\nu=\la d\cap\mu$.
Then for each $\bfa(\nu)\in \La(\nu,<\!m)$, there is
a unique $\da\in \cnrm_{\la, \mu}^d$ such that $\sigma^{\bfa(\nu)}=\sigma^{\ddot{\da}}$.
\end{Lem}

\begin{proof} Let $A=(a_{ij})=\theta(\la,d,\mu)\in \Theta(n,r)$, i.~e.,
$(\row(A),\, d_A, \,\col(A))=(\la,d,\mu)$. Then
$$\nu=\la d\cap\mu=\nu(A)=(a_{11}, \ldots, a_{n1},a_{12}, \ldots, a_{n2}, \ldots, a_{1n}, \ldots, a_{nn} )\in \Lambda(n^2, r).$$
Applying the bijections $g_{m,a_{ij}}$ defined as in \eqref{bijection-lamda-gamma}
induces a bijection
$$g_A:\cnrm_{\la, \mu}^d\longrightarrow \La(\nu(A),<\!m),\; \da\longmapsto \bfa(\ddot\da),$$
noting that $|\da|=A$. Hence, $\da=g_A^{-1}\big(\bfa(\nu)\big)$.
\end{proof}

The lemma above together with Proposition \ref{symm-funct-composition}(1) gives the following result.

\begin{Prop}\label{linear independent 0} For $\la, \mu\in \Lambda(n,r)$ and
$d\in \msD_{\la, \mu}$, the set
 $$\{\sigma^{\ddot\da}\mid \da \in \Theta_m(n,r)_{\la, \mu}^d\}$$
is linearly independent. Moreover, for each $\da\in \cnrm$,
\begin{equation} \label{comm-x-nu-sigma}
x_{\nu(d_{|\da|})}\sigma^{\ddot\da}=\sigma^{\ddot\da}x_{\nu(d_{|\da|})}.
\end{equation}
\end{Prop}

\begin{Example} Let $n=2$, $m=3$, and $r=11$. Choose
$$\da=
  \begin{pmatrix}
    (1, 1, 1 ) &(1, 0, 2) \\
    (1, 1, 0) & (1, 2, 0) \\
  \end{pmatrix}\in {\Theta}_3(2,11).$$
Then $\bfa_{11}=(1,1,1)$, $\bfa_{21}=(1,1,0)$, $\bfa_{12}=(1,0,2)$, $\bfa_{22}=(1,2,0)$, and
$$|\da|=
  \begin{pmatrix}
    3 &3 \\
    2 & 3\\
  \end{pmatrix}.$$
By the definition above, we obtain
$$\aligned
\ddot{\bfa}_{11}&=(1, 1, 0), \;\ddot{\bfa}_{21}=(1, 1), \; \ddot{\bfa}_{12}=(1, 0, 0),
\;\ddot{\bfa}_{22}=(1, 0, 1),\;\text{ and}\\
\sigma^{\ddot\da}&=\sigma^{\ddot{\bfa}_{11}}\sigma^{\ddot{\bfa}_{21}}\sigma^{\ddot{\bfa}_{12}}\sigma^{\ddot{\bfa}_{22}}\\
&=(L_1+L_2+L_3)(L_1L_2+L_2L_3+L_1L_3)(L_4+L_5)(L_4L_5)\\
&\quad\; (L_6+L_7+L_8)(L_9+L_{10}+L_{11})L_9L_{10}L_{11}.
\endaligned$$
Moreover, $\nu(d_{|\da|})=(3, 2, 3, 3)$, and the corresponding Young subgroup $\fS_{\nu(d_{|\da|})}$
is generated by $s_1, s_2, s_4, s_6, s_7, s_9, s_{10}$, where $s_i$ ($1\leq i\leq 10$)
are the generators of $\fS_{11}$.
\end{Example}

Using the element $\sigma^{\ddot\da}$ associated with $\da\in \cnrm$ in \eqref{sigddA}, we further define with $\la=\ro(|\da|)$ and $ \mu=\co(|\da|),$
\begin{equation}\label{def-b_A}
\frak b_{\da}=x_\la T_{d_{|\da|}} \sigma^{\ddot\da}\sum_{w\in \msD_{\nu(d_{|\da|})}\cap \fS_\mu}T_w
\in x_\la \cysHr.
\end{equation}

\begin{Thm}\label{standard-basis-thm} For $\lambda, \mu\in \Lambda(n,r)$, the $\scR$-module
$x_\la\cysHr \cap \cysHr  x_\mu$ is free and the set
$$\mathcal{B}_{\la,\mu}=\mathcal{B}^{(m)}_{\la,\mu}=\big\{\frak b_{\da}\mid \da\in \cnrm_{\la, \mu} \big\}$$
forms a basis.
\end{Thm}

The proof of the theorem is somewhat standard and will be given in Appendix A; compare \cite{DW}.

We remark that the assertion that $\scR$-module  $x_\la\cysHr \cap \cysHr  x_\mu$ is $\scR$-free is not new.
In \cite[(6.3)]{DJM}, a basis that can give rise to a cellular basis for the cyclotomic $q$-Schur algebra is
constructed. However, the basis $\mathcal{B}_{\la,\mu}$ is new. We will show in \S5 that it is a $q$-analogue
of the usual double coset basis for the endomorphism algebra of a permutation module \cite{S}.

For given $\da\in\cnrm$, define
$$\Phi_{\da}\in\sS_\bfu(n,r)=\End_{\cysHr }\big(\bigoplus_{\la\in\La(n,r)}x_\la \cysHr \big)$$
taking $x_\mu h\mapsto \delta_{\mu, \col(|\da|)}\frak b_{\da}h$
for all $\mu\in \Lambda(n,r) $ and $h\in \cysHr $.

Now we can state the main result of this section.

\begin{Thm}\label{thm-standard-basis}
The slim cyclotomic $q$-Schur algebra $\sS_\bfu(n,r)$ is a free $\scR$-module with basis
$\{\Phi_{\da}\mid \da\in\cnrm\}$ and rank
$\left( \begin{matrix} mn^2+r-1\\ r \end{matrix}\right) $.
\end{Thm}

\begin{proof} Since $\sS_\bfu(n,r)=\bigoplus_{\la, \mu\in \Lambda(n,r)} \Hom_{ \cysHr }(x_{\mu} \cysHr , x_{\la} \cysHr)$, the first assertion follows from Lemma \ref{permutation mod} and Theorem \ref{standard-basis-thm}. The second assertion is clear as
$|\cnrm|=\left( \begin{matrix} mn^2+r-1\\ r\end{matrix}\right)$.
\end{proof}

\begin{Rem}\label{aff3} (1)
It would be interesting to generalise the construction to obtain a similar basis for the affine
$q$-Schur algebra $\afsSr$. However, by simply replacing $L_j$'s by $X_j$'s in the definition, we obtain a linearly independent set
$\{\Phi_{\da}^\vtg\mid \da\in\cnrm\}$ for $\afsSr$. See Remarks \ref{aff1} and \ref{aff2}.

(2) We also observe that the basis in Theorem \ref{thm-standard-basis} cannot be naturally extended to a basis for the cyclotomic $q$-Schur algebra discussed in \cite{DJM}.
\end{Rem}

View $\Theta(n,r)$ as a subset of $\cnrm$ via the following map:
\begin{equation}\label{iota}
\iota^{(m)}:\Theta(n,r)\lra\cnrm,\quad A\longmapsto A^{(m)},
\end{equation}
where if $A=(a_{ij})\in\Theta(n,r)$, then $A^{(m)}=(\bfa_{ij})\in\cnrm$ with all
$\bfa_{ij}=(0, \ldots, 0, a_{ij})\in\mbn^m$. Then by the definition,
$\sigma^{\ddot A^{(m)}}=1$ and, hence,
$$\Phi_{A^{(m)}}(x_\mu h )=\delta_{\mu, \col(A)}x_\la T_{d_{A}} \sum_{v\in \msD_{\nu(d_{A})}\cap \fS_\mu}T_vh,
\;\forall\,h\in\cysHr.$$
Note that the same rule with $h\in\sH(r)$ defines an element $\phi_A$ in the $q$-Schur algebra $\sS_q(n,r):=\End_{\sH(r)}(\oplus_{\la\in\La(n,r)}x_\la \sH(r))$. Thus, we have immediately the following:

\begin{Prop}\label{q-Sch} The $q$-Schur algebra $\sS_q(n,r)$ can be embedded
into $\sS_\bfu(n,r)$ via the map $\phi_A\mapsto\Phi_{A^{(m)}}$ for $A\in \Theta(n,r)$.
\end{Prop}

\section{Comparison with the double coset basis of type $B/C$}

Recall the presentation of $\fS_{m,r}$ in \eqref{presentn}. If $m=2$, then
$$W:=\fS_{2,r}$$ is a Coxeter (or Weyl) group of type $B$ with Coxeter generators $s_0,s_1,\ldots,s_{r-1}$. In this section we mainly deal with the cyclotomic Hecke algebra $\cysHr$ associated with $W$. We will also choose
$$\bfu={\bf 2}=(-1, q_0)$$ with $q_0$ invertible in $\scR$ ($q$ also invertible
in $\scR$ as before) so that $\tbsHr$
is isomorphic to the Hecke algebra of type $B$ defined in \cite[Sect.~3]{DS} by sending $T_i$ to $T_{s_i}$ and $L_j$ to $q^{1-j}T_{s_{j-1}}\cdots T_{s_1}T_{s_0}T_{s_1}\cdots T_{s_{j-1}}$. We will identify the two. Thus, $T_i=T_{s_i}$, $L_1=T_0$, and $\tbsHr$ has a basis
$\{T_w\}_{w\in W}$, where $T_w=T_{i_1}T_{i_2}\cdots T_{i_l}$ for any reduced expression $w=s_{i_1}s_{i_2}\cdots s_{i_l}$. The associated slim cyclotomic $q$-Schur algebra is known as the {\it $q$-Schur$^{1\textsc{b}}$ algebra}.\footnote{The (nonslim) cyclotomic $q$-Schur algebra in this case is called the $q$-Schur$^{2\textsc{b}}$ (or $q$-Schur$^2$) algebra; see \cite{DS}.} See Remark \ref{octah} for connections to other similar algebras.

For any subset $D$ of $ W$, define
               $$T_D=\sum_{w\in D}T_w\in\tbsHr,\quad \underline{D}=\sum_{w\in D}w\in\scR W.$$

For arbitrary $\la,\mu\in\La(n,r)$, we shall make a comparison between the integral basis
$\mathcal B_{\la,\mu}^{(2)}$ for $x_\la\tbsHr \cap \tbsHr  x_\mu$
given in Theorem \ref{standard-basis-thm} and the double coset basis $\{T_D\mid D\in\fS_\la\backslash W/\fS_\mu\}$
given in \cite[Prop.~4.2.5]{DS} (and their counterparts in the non-quantum case).

Following \cite[Sect.~2]{DS}, for each composition $\lambda$ of $r$, let $\widetilde{\msD}_\lambda$
be the set of shortest right coset representatives of the Young subgroup $\fS_\lambda$ in $ W$.
Set $\widetilde{\msD}_\lambda^{-1}=\{d^{-1}\;|\; d\in \wdl\}$.  Then for two compositions
$\la,\mu$ of $r$, $\wdlm=\wdl\cap \widetilde{\msD}_\mu^{-1}$ is the set of shortest $\fS_\lambda$-$\fS_\mu$ double coset representatives in $ W$.

We will prove the following theorem in this section.

\begin{Thm}\label{double coset basis}
For $\la, \mu\in \Lambda(n,r)$, let $\mathcal B_{\la,\mu}^{(2)}$ be the basis for the free
$\scR$-module $x_\la\tbsHr \cap \tbsHr  x_\mu$ as given in Theorem \ref{standard-basis-thm}.
\begin{itemize}
\item[(1)] For $\la=\mu=(r)=(r,0,\ldots,0)\in\La(n,r)$, write $\widetilde{\msD}_{r,r}:=\widetilde{\msD}_{\la,\la}$. Then $$\mathcal B_{r,r}^{(2)}=\{T_{\fS_r d\fS_r}\mid d\in\widetilde{\mathscr D}_{r,r}\}.$$
\item[(2)] If $\sH_{\bf 2}(r)$ is the group algebra of $W$ over $\scR$ (i.~e., if $q=q_0=1$), then
$$\mathcal B_{\la,\mu}^{(2)}=\big\{\underline{\fS_\la d\fS_\mu} \mid d\in \wdlm\big\}.$$
\end{itemize}
\end{Thm}
\begin{Rem}
By this result, we see the basis constructed in Theorem \ref{thm-standard-basis} for the (non-quantum)
Schur$^{1\textsc b}$ algebra coincides with the standard basis for the endomorphism algebra of a permutation module \cite{S}.
Thus, both the basis in Theorem \ref{thm-standard-basis} and the basis %$\{T_D\mid D\in\fS_\la\backslash W/\fS_\mu\}$
given in \cite[Thm 4.2.6]{DS} for the $q$-Schur$^{1\textsc b}$ algebra are quantum analogues of
the classical double coset basis for the Schur$^{1\textsc b}$ algebra.

Thus, the basis constructed in
Theorem \ref{thm-standard-basis} can be regarded as a generalisation of the double coset basis to the Hecke algebra of a complex reflection group.
\end{Rem}

We proceed to prove the theorem.

Let $d_0=1$ and, for $1\leq i\leq r$,  let
\begin{equation}\label{d_i}
d_i=\tau_1\cdot\tau_2\cdots \tau_i,\quad\text{ where }\;\; \tau_j=s_{j-1}s_{j-2}\ldots s_1 s_0.
\end{equation}

\begin{Lem}\label{shortest} We have
$\widetilde{\msD}_{r,r}=\{d_i\mid 0\leq i\leq r\}$ and $\ell(d_i)=i(i+1)/2$.
\end{Lem}

\begin{proof}This is clear since the double coset $\fS_r d_i\fS_r$ is the subset of $\fS_{2,r}$ consisting of elements satisfying the property that  $s_0$ occurs exactly $i$ times in every reduced expression.
\end{proof}

\begin{Lem}\label{right coset} For $0\leq i\leq r$, we have
$$\fS_{\nu(d_i)}:=d_i^{-1}\fS_r d_i\cap \fS_r=\begin{cases}\fS_r,&\text{ if }i=0;\\
\fS_{\{1,2,\ldots, i\}}\times\fS_{\{i+1,i+2,\ldots,r\}},&\text{ if }i\geq1.\end{cases}$$
\end{Lem}

\begin{proof} We first observe that $\fS_{\nu(d_1)}$ is the subgroup of $\fS_r$ generated by
$\{s_2, s_3, \ldots , s_{r-1}\}$. If $2\leq i\leq r$, then, for $1\leq j\leq r-1$,
\begin{equation}\label{right coset equation}
d_i^{-1}s_j d_i=\begin{cases} s_{i-j}, \; \;& \text{if $j\leq i-1$};\\
s_0s_1\ldots s_{i-1}s_is_{i-1}\ldots s_1s_0, \;\;&\mbox{if $j=i$};\\
s_j,\;\;&\text{if $j\geq i+1$}.
\end{cases}
\end{equation}
Therefore, $\fS_{\nu(d_i)}$ is generated by $\{s_1, s_2, \ldots , s_{i-1}, s_{i+1}, \ldots s_{r-1}\}$,
as desired.
\end{proof}

We now compute $\msD_{\nu(d_i)}=\widetilde{\msD}_{\nu(d_i)}\cap \fS_r$. For $2\leq i\leq r$, let
$$\mathcal{E}_i=\{{\bf j}=(j_1, \ldots, j_i)\in\mbn^i\mid 1\leq j_1<\ldots<j_i\leq r\}.$$
Then, for each ${\bf j}\in \mathcal{E}_i$, we have $j_k\geq k$ for all $1\leq k\leq i$.

\begin{Lem}\label{right coset1} We have $\msD_{\nu(d_0)}=\{1\}$, $\msD_{\nu(d_1)}=\{1, s_1, s_1s_2, \ldots, s_1s_2\ldots s_{r-1}\}$,
and, for $2\leq i\leq r$,
$$\msD_{\nu(d_i)}=\{v_{\bf j}=(s_i\ldots s_{j_i-1})(s_{i-1}\ldots s_{j_{i-1}-1})\ldots
(s_1\ldots s_{j_1-1})\mid {\bf j}=(j_1,\ldots,j_i)\in\mathcal{E}_i\}.$$
{\rm(}If $j_k=k$, set $s_k\ldots s_{j_k-1}=1$ by convention.{\rm)}
\end{Lem}

\begin{proof} Let $v_{\bf j}^{-1}=(s_{j_1-1}\cdots s_1)(s_{j_2-1}\cdots s_2)\cdots (s_{j_i-1}\cdots s_i)$
act on the ordered sequence $1,\ldots, r$. Then we get an ordered sequence such that the first $i$ numbers are
$j_1, \ldots, j_i$ and the remaining ones are in increasing order if and only if the number of inversions
of the sequence is $\sum_{1\leq k\leq i}(j_k-1)$, which equals the
length of $v_{\bf j}^{-1}$. Hence, $v_{\bf j}^{-1}$, as well as $v_{\bf j}$,
is reduced. The lemma then follows from the equalities $|\mathcal{E}_i|=|\msD_{\nu(d_i)}|=\frac{r!}{i!(r-i)!}$
 and the fact that $v_{\bf j}\neq v_{\bf j'}$ whenever ${\bf j}\neq {\bf j'}$.
\end{proof}

Thus, the double coset basis of $x_{(r)}\tbsHr \cap \tbsHr  x_{(r)}$ consists of the following
 elements
$$x_{(r)},\quad x_{(r)}T_0(T_1+T_1T_2+\cdots +T_1T_2\cdots T_{r-1}),\quad x_{(r)}T_{d_i}
\sum_{{\bf j}\in\mathcal{E}_i}T_{v_{\bf j}},\quad 2\leq i\leq r.$$

By taking $\mathbf m={\bf 2}=(-1,q_0)$, the basis for $x_{(r)}\tbsHr \cap \tbsHr  x_{(r)}$ in Proposition \ref{basis 1} becomes
$$\mathcal B^{(2)}_{r,r}=\{x_{(r)}\sigma_i\;|\; 0\leq i\leq r\}\quad(\text{where }\sigma_0=1).$$

\begin{proof}[{\bf Proof of Theorem \ref{double coset basis}(1)}]
It is obvious for $i=0$ as $x_{(r)}=T_{\fS_r}$.
 Since $x_{(r)}T_i=qx_{(r)}$ for $1\leq i\leq r-1$ and $L_i=q^{-(i-1)}T_{i-1}T_{i-2}\cdots T_1T_0T_1\cdots T_{i-1}$, it follows from Lemmas \ref{right coset1} and \ref{Basic}(2) that
$$\aligned
x_{(r)}\sigma_1&=x_{(r)}(L_1+L_2+\ldots +L_r)\\
&=x_{(r)}T_0(1+T_1+T_1T_2+\ldots +T_1T_2\ldots T_{r-1})=x_{(r)}T_0T_{\widetilde\msD_{\nu(d_1)}\cap\fS_r}
=T_{\fS_rd_1\fS_r}.
\endaligned$$

We now assume $2\leq i\leq r$. We first claim that for ${\bf j}=(j_1,\ldots,j_i)\in\mathcal{E}_i$,
$$x_{(r)}T_{d_i}T_{v_{\bf j}}=x_{(r)}L_{\bf j}\;\;\text{ where }\;\;L_{\bf j}=L_{j_1}L_{j_2}\ldots L_{j_i}.$$
Indeed, by repeatedly applying \eqref{right coset equation} or a well-known fact, we have
$$
t_1t_2\cdots t_i=s_1(s_2s_1)\cdots (s_{i-1}\cdots s_1)d_i=w_{0,i}d_i,$$
where $w_{0,i}$ is the longest element of $\fS_{\{1,2,\ldots,i\}}$, and $t_1, \ldots, t_i$ are defined in \eqref{tijk319}.
Thus,  $x_{(r)}T_{d_i}=x_{(r)}L_1L_2\cdots L_i$.
By \eqref{LiTi}, for $2\leq i\leq r-1$,
$$\aligned
x_{(r)}L_1\ldots L_{i-1}&L_iT_{i}=x_{(r)}L_1L_2\ldots L_{i-1} (L_iT_{i})\\
&=x_{(r)}L_1L_2\ldots L_{i-1} \big(T_{i}L_{i+1}+(1-q)L_{i+1}\big)\\
&=x_{(r)}\big(T_iL_1L_2\ldots L_{i-1}L_{i+1}+(1-q)L_1L_2\ldots L_{i-1}L_{i+1}\big)\\
&=x_{(r)}\big(qL_1L_2\ldots L_{i-1}L_{i+1}+(1-q)L_1L_2\ldots L_{i-1}L_{i+1}\big)\\
&=x_{(r)}L_1L_2\ldots L_{i-1}L_{i+1}.
\endaligned$$
Hence, inductively, we obtain
$$x_{(r)}L_1L_2\ldots L_i(T_{i}T_{i+1}\ldots T_{j_i-1})=x_{(r)}L_{1}\ldots L_{i-1}L_{j_{i}},
\;\forall\, 2\leq i<j_i\leq r-1.$$
Since $T_{v_{\bf j}}=(T_{i}T_{i+1}\ldots T_{j_i-1})(T_{i-1}T_i\ldots T_{j_{i-1}-1})\ldots(T_1T_2\ldots T_{j_1-1})$,
repeatedly applying the formula gives
 \begin{equation}\label{right action}
 \aligned
x_{(r)}T_{d_i}T_{v_{\bf j}}
&=x_{(r)}L_{j_1}L_{j_2}\ldots L_{j_i}
\endaligned
\end{equation}
proving the claim.
%Recall that  $$T_{\fS_rd_i\fS_r}=x_{(r)}T_{d_i}T_{\widetilde{\msD}_{\nu(d_i)}\cap\fS_r}$$ by \cite[Prop. 4.2.2]{DS}.
Now, by Lemmas \ref{Basic}(2) and  \ref{right coset1},  and the claim,
\begin{equation}\label{coincide 1}T_{\fS_rd_i\fS_r}=x_{(r)}T_{d_i}T_{\widetilde{\msD}_{\nu(d_i)}\cap\fS_r}=\sum_{{\bf j}\in\mathcal{E}_i}x_{(r)}T_{d_i}T_{v_{\bf j}}
=x_{(r)}\sum_{{\bf j}\in\mathcal{E}_i}L_{\bf j}=x_{(r)}\sigma_i,
\end{equation}as desired.
\end{proof}

 For $a, b\in\mbn$ with $b\geq1,a+b\leq r$, define  the subgroups $W^a_b,\fS^a_b$ of $W=\fS_{2,r}$:
\begin{equation}\label{Wab}
\aligned
W^a_b&=\langle s_{a+1}, s_{a+2}, \ldots, s_{a+b-1},\; t_{a+1}\rangle,\\
\fS^a_b&=\langle s_{a+1}, s_{a+2}, \ldots, s_{a+b-1}\rangle,\endaligned
\end{equation}
where $t_{a+1}=s_as_{a-1}\cdots s_1s_0s_1\cdots s_{a-1}s_a$. Note that $W^a_b$ is not a parabolic subgroup of $W$.
However, since $t_{a+1}^2=1$ and $t_{a+1}s_{a+1}t_{a+1}s_{a+1}=s_{a+1}t_{a+1}s_{a+1}t_{a+1}$, there is a group embedding
\begin{equation}\label{parallel 1}
\iota:\fS_{2,b}\lra W,\;\; s_0\longmapsto t_{a+1}, s_i\longmapsto s_{a+i}\;(1\leq i\leq b-1),
\end{equation}
which induces a group isomorphism $\fS_{2,b}\cong W^a_b$.

Note that $\fS^a_b$ is a parabolic subgroup of $W$ and $\fS^a_b=\fS_{\mu_{a,b}}$, the Young subgroup associated
with $\mu_{a,b}={(1^a,b,1^{r-a-b})}$. Set
$$\widetilde\msD^a_{b,b}=\widetilde\msD_{\mu_{a,b},\mu_{a,b}}\cap W^a_b,$$
the set of shortest double coset representatives of $\fS^a_b$ in $W^a_b$.
By \eqref{parallel 1} and Lemma \ref{shortest}, we have the following result.

\begin{Lem}\label{shortest parallel} For $a,b\in\mbn$ with $b\geq1, a+b\leq r$, we have
$$\widetilde\msD^a_{b,b}=\{d_{a,b}^{(0)}=1,\;\; d_{a,b}^{(i)}=\tau_1'\tau_2'\tau_3'\ldots \tau_i'\mid 1\leq i\leq b\},$$
where $\tau_t'=s_{a+t-1}s_{a+t-2}\ldots s_{a+1}t_{a+1}$ for $2\leq t\leq b$ and $\tau_1'=t_{a+1}$.
\end{Lem}

\begin{proof}
By Lemma \ref{shortest}, the elements $1, \; \tau_1'\tau_2'\cdots\tau_i'$ $(1\leq i\leq b)$ form a complete
set of $(\fS^a_b,\fS^a_b)$-double coset representatives in $W^a_b$. Since, for every $s_{a+i}$ $(1\leq i<b)$,
the product $s_{a+i}\tau_1'\tau_2'\cdots\tau_i'$ has length increased, we conclude that
$$\widetilde\msD^a_{b,b}=\{1, \; \tau_1'\tau_2'\cdots\tau_i'\mid 1\leq i\leq b\}.$$
\end{proof}

If there is no confusion, we often drop the subscripts $a,b$ in $d_{a,b}^{(i)}$.
For $0\leq i\leq b$, let $\fS^a_{b,\nu(d^{(i)})}={d^{(i)}}^{-1}\fS^a_b d^{(i)}\cap \fS^a_b$
and let $\msD^a_{b,\nu(d^{(i)})}$ be the set of shortest right coset representatives of
$\fS^a_{b,\nu(d^{(i)})}$ in $\fS^a_b$.

Set $x^a_b=\sum_{w\in \fS_b^a}T_w$. For $1\leq i\leq b$, let $\sigma^a_{b,i}$ be
the $i$-th elementary symmetric polynomial in $L_{a+1}, \ldots, L_{a+b}$ and set $\sigma^a_{b,0}=1$. By
a similar argument at the end of the proof of Lemma \ref{symm-polynomials}, we have
\begin{equation}\label{comm}
\sigma^a_{b,i}T_j=T_j\sigma^a_{b,i},\quad \forall\, 0\leq i\leq b,\; a+1\leq j<a+b-1.
\end{equation}
The following lemma is an analogue to \eqref{coincide 1}.

\begin{Lem}\label{coincide 2} Let $a,b\in\mbn$ with $b\geq1,a+b\leq r$. Then for each $0\leq i\leq b$,
$$x^a_b\sigma^a_{b,i}=q^{-ai}T_{\fS^a_bd^{(i)}\fS^a_b}=q^{-ai}x^a_bT_{d^{(i)}}\bigg(\sum_{w\in \msD^a_{b,\nu(d^{(i)})}}T_w\bigg).$$
\end{Lem}

\begin{proof}Let $\sH(\fS_{2,b})$ be the Hecke algebra associated with the Coxeter group with basis $T'_w,w\in\fS_{2,b}$.
The map $\iota$ in \eqref{parallel 1} induces an $\scR$-module embedding
$$\widetilde\iota:\sH(\fS_{2,b})\lra \sH_{\mathbf 2}(r), T_w'\longmapsto T_{\iota(w)}.$$
This map satisfies the property: if $w=s_{i_1}\cdots s_{i_l}$ is a reduced expression, then
$\widetilde{\iota}(T_w')=T_{\iota(w)}=T_{\iota(s_{i_1})}\cdots T_{\iota(s_{i_l})}=\widetilde{\iota}(T_{s_{i_1}}')\cdots \widetilde{\iota}(T_{s_{i_l}}')$,
as $\iota(s_{i_1})\cdots\iota(s_{i_l})$ is reduced. Since $\widetilde{\iota}(L'_i)=q^{a}L_{a+i}$,
applying $\widetilde{\iota}$ to the formula in \eqref{coincide 1} for $\sH(\fS_{2,b})$ yields the required one.
\end{proof}

Consider a basis element $\frak b_{\da}\in \mathcal{B}_{\la\mu}^{(2)}$ with $\da=(\bfa_{ij})\in \Theta_2(n,r)_{\la, \mu}$, where $\bfa_{ij}=(a_{ij}^{(1)},a_{ij}^{(2)})$. Then
$\sJ=\{(i,j)\in\sI\mid a_{ij}^{(1)}\neq0\}$; see \eqref{sIsJ}.
Note that   $ \fS_{\nu(\da)}\subseteq \fS_{\nu(|\da|)}\subseteq \fS_{\co(|\da|)}\subseteq \fS_r$. Thus,
\begin{equation}\label{DDD}
\msD_{\nu(\da)}\cap \fS_\mu=(\msD_{\nu(\da)}\cap \fS_{\nu(|\da|)})(\msD_{\nu(|\da|)}\cap \fS_\mu).\end{equation}
Every number $\widetilde a_{ij}$ defined in \eqref{aijk} defines a subgroup $W^{\widetilde a_{ij}}_{|\bfa_{ij}|}$ of $W$ (by convention, $W^{\widetilde a_{ij}}_{|\bfa_{ij}|}=1$ if $|\bfa_{ij}|=0$), as well as the
 subgroup $\fS^{\widetilde a_{ij}}_{|\bfa_{ij}|}$ of $W^{\widetilde a_{ij}}_{|\bfa_{ij}|}$.
By applying Lemma \ref{shortest parallel} to the case $a=\widetilde a_{ij}$ and $b=|\bfa_{ij}|$,
we obtain an element
$$d_{ij}:=d_{\widetilde a_{ij}, |\bfa_{ij}|}^{(a_{ij}^{(1)})}\in
\widetilde\msD^{\widetilde a_{ij}}_{|\bfa_{ij}|, |\bfa_{ij}|}.$$

\begin{Prop}\label{coincide 3} For arbitrary
$\la, \mu\in \Lambda(n,r)$, $\da\in\cnrm_{\la\mu}$ and $\sJ=\sJ(\da)$, we have
$$
\frak b_{\da}=\bigg(\prod_{(i,j)\in\sJ} q^{-\widetilde a_{ij}a_{ij}^{(1)}}\bigg)x_\la T_{d_{|\da|}}\prod_{(i,j)\in\sJ}^{(\preceq)} T_{d_{ij}}\sum_{w\in \msD_{\nu(\da)}\cap \fS_\mu}T_w,$$
where the product of $T_{d_{ij}}$ is taken over the order $\preceq$ given in \eqref{prec}.
\end{Prop}

\begin{proof}Let $d=d_{|\da|}$. Then $\nu(d)=\nu(|\da|)$ satisfies $\fS_{\nu(|\da|)}=d^{-1}\fS_\la d\cap\fS_\mu$.
 By \eqref{element} and \eqref{def-b_A},
\begin{equation}\label{bij t1}
\frak b_{\da}=x_\la T_{d} \sigma^{\ddot\da}\sum_{v\in \msD_{\nu(d)}\cap \fS_\mu}T_v
=\sum_{u\in \msD^{-1}_{\nu(d^{-1})}\cap \fS_{\la}}T_uT_{d}x_{\nu(d)} \sigma^{\ddot\da}\sum_{v\in \msD_{\nu(d)}\cap \fS_\mu}T_v.
\end{equation}
Then, by the definition,
$$
x_{\nu(d)}=\prod_{i, j=1}^n x^{\widetilde a_{ij}}_{|\bfa_{ij}|}.$$
 Moreover, the terms $x^{\widetilde a_{ij}}_{|\bfa_{ij}|}$ commute with each other.
Also, by definition,
\begin{equation}\label{bij t2}
x_{\nu(d)}\sigma^{\ddot\da}=\prod_{i, j=1}^n x^{\widetilde a_{ij}}_{|\bfa_{ij}|}\sigma^{\ddot{\bfa}_{ij}}.
\end{equation}
Since $\ddot{\bfa}_{ij}\in \La(|\bfa_{ij}|,<\!2)$ is computed by the partition dual to $(a_{ij}^{(1)})$, it follows that
$$\ddot{\bfa}_{ij}=\begin{cases}(0,\ldots,0, 1, 0, \ldots,0),&\text{if }a_{ij}^{(1)}\not=0;\\
(0,\ldots,0),&\text{otherwise,}\end{cases}$$
 where the $1$ is at the $a_{ij}^{(1)}$-th position.

Note that, if $(i,j)\in\sJ$, then
$\sigma^{\ddot{\bfa}_{ij}}=\sigma^{\widetilde a_{ij}}_{|\bfa_{ij}|,a_{ij}^{(1)}}$
is the ${a_{ij}^{(1)}}$-th elementary symmetric polynomial in
$L_{\widetilde a_{ij}+1}, L_{\widetilde a_{ij}+2}, \ldots, L_{\widetilde a_{ij}+|\bfa_{ij}|}.$
  We set
$\sigma^{\ddot{\bfa}_{ij}}=1$ for $(i,j)\not\in\sJ$.
By Lemma \ref{coincide 2}, for $(i,j)\in\sJ$,
$$x^{\widetilde a_{ij}}_{|\bfa_{ij}|}\sigma^{\ddot{\bfa}_{ij}}
=x^{\widetilde a_{ij}}_{|\bfa_{ij}|}\sigma^{\widetilde a_{ij}}_{|\bfa_{ij}|,a_{ij}^{(1)}}
=q^{-\widetilde a_{ij}a_{ij}^{(1)}}\cdot x^{\widetilde a_{ij}}_{|\bfa_{ij}|}T_{d_{ij}}\sum_{x\in \msD_{ij}} T_x,$$
where $\msD_{ij}=\msD^{\widetilde a_{ij}}_{|\bfa_{ij}|, \nu(d_{ij})}$ is the set of
shortest right coset representatives of
$$\fS_{\nu(d_{ij})}=\fS_{(1^{\widetilde a_{ij}},\bfa_{ij},1,\ldots)}=d_{ij}^{-1}\fS^{\widetilde a_{ij}}_{|\bfa_{ij}|}d_{ij}\cap \fS^{\widetilde a_{ij}}_{|\bfa_{ij}|}\quad
\text{ in }\;\;\fS^{\widetilde a_{ij}}_{|\bfa_{ij}|}.$$
By substituting into \eqref{bij t2}, we obtain that
 $$\aligned
 x_{\nu(d)} \sigma^{\ddot\da}&= \prod_{(i, j)\in\sI}^{(\preceq)} \bigg(q^{-\widetilde a_{ij}a_{ij}^{(1)}}\cdot x^{\widetilde a_{ij}}_{|\bfa_{ij}|}
 T_{d_{ij}}\big(\sum_{x\in \msD_{ij}} T_x\big)\bigg)\\
&=\bigg(\prod_{(i, j)\in\sJ} q^{-\widetilde a_{ij}a_{ij}^{(1)}}\bigg) x_{\nu(d)}\prod_{(i, j)\in\sJ}^{(\preceq)} T_{d_{ij}}\prod_{(i,j)\in\sJ}^{(\preceq)} \big(\sum_{x\in \msD_{ij}} T_x\big),\\
\endaligned$$
where $(\preceq)$ indicates the products are taken over the order $\preceq$ defined in \eqref{prec}.
The second equality is seen from the fact that if $(k,l)\preceq(i,j)$, then $ T_{d_{kl}}\big(\sum_{x\in \msD_{kl}} T_x\big)$ commutes with $x^{\widetilde a_{ij}}_{|\bfa_{ij}|}$, and  $\sum_{x\in \msD_{kl}} T_x$ with $T_{d_{ij}}$.
 Then \eqref{bij t1} becomes
 \begin{equation}\label{bij t4}
\aligned\frak b_{\da}&=\sum_{u\in \msD^{-1}_{\nu(d^{-1})}\cap \fS_{\la}}T_uT_{d}x_{\nu(d)}
\sigma^{\ddot\da}\sum_{v\in \msD_{\nu(d)}\cap \fS_\mu}T_v\\
&=\sum_{u\in \msD^{-1}_{\nu(d^{-1})}\cap \fS_{\la}}T_uT_{d}x_{\nu(d)}\prod_{i, j=1}^nq^{-\widetilde a_{ij}a_{ij}^{(1)}}
T_{d_{ij}}\prod_{i,j=1}^n\big(\sum_{x\in \msD_{ij}} T_x\big)\sum_{v\in \msD_{\nu(d)}\cap \fS_\mu}T_v\\
&=\bigg(\prod_{i,j=1}^n q^{-\widetilde a_{ij}a_{ij}^{(1)}}\bigg)x_\la T_d\prod_{i,j=1}^n T_{d_{ij}}\prod_{i,j=1}^n \big(\sum_{x\in \msD_{ij}}T_x\big)
\sum_{v\in \msD_{\nu(d)}\cap \fS_\mu}T_v.\\
\endaligned
\end{equation}
By Lemma \ref{7 crucial},
 $$\prod_{i, j=1}^n\big(\sum_{x\in \msD_{ij}}T_x\big)=\sum_{w\in \msD_{\nu(\da)}\cap \fS_{\nu(|\da|)}}T_w,$$
it follows by \eqref{DDD} that
 \begin{equation*}\label{bij t5}
\aligned
{}\prod_{i, j=1}^n\big(\sum_{x\in \msD_{\nu(d_{ij})}} T_x\big)\sum_{v\in \msD_{\nu(d)}\cap \fS_\mu}T_v
=\sum_{w\in \msD_{\nu(\da)}\cap \fS_{\nu(|\da|)}}T_w\sum_{v\in \msD_{\nu(|\da|)}\cap \fS_\mu}T_v
=\sum_{w\in \msD_{\nu(\da)}\cap \fS_\mu}T_w.
\endaligned
\end{equation*}
Substituting into  \eqref{bij t4} gives the desired formula.
\end{proof}
\begin{proof}[\bf Proof of Theorem \ref{double coset basis}(2)] Let $w$ be the element given in \eqref{w=dt}
and assume that $d_\da$ is the shortest representative of $\fS_\la w\fS_\mu$. We claim that if
$\ddot{d}=d\cdot\prod_{(i,j)\in\sJ}^{(\preceq)} d_{ij}$, then $\fS_\la\ddot{d}\fS_\mu=\fS_\la w\fS_\mu$.
Indeed, for $a=\widetilde{a}_{ij}$, $b=|a_{ij}|$ and $c=a_{ij}^{(1)}>0$, we have $d_{ij}=1=t_{ij}$
if $c=0$, $d_{ij}=t_{a+1}=t_{ij}$ if $c=1$, and, for $2\leq c\leq b$
$$\aligned
d_{ij}&=t_{a+1}\cdot(s_{a+1}t_{a+1})\cdot(s_{a+2}s_{a+1}t_{a+1})\cdot\cdots (s_{a+c-1}s_{a+c-2}\cdots s_{a+1}t_{a+1})\\
&=t_{a+1}\cdot(t_{a+2}s_{a+1})\cdot(t_{a+3}s_{a+2}s_{a+1})\cdot\cdots (t_{a+c}s_{a+c-1}s_{a+c-2}\cdots s_{a+1})\\
&=t_{ij}\cdot s_{a+1}\cdot (s_{a+2}s_{a+1})\cdot(s_{a+3}s_{a+2}s_{a+1})\cdots (s_{a+c-1}s_{a+c-2}\cdots s_{a+1}),
\endaligned
$$
(thus, $d_{ij} d_{kl}=d_{kl} d_{ij}$) proving the claim. By Proposition \ref{coincide 3}, we have in the group algebra
$\frak b_{\da}=\underline{\fS_\la} \cdot{\ddot d}\cdot\underline{\msD_{\nu(\da)}\cap \fS_\mu}
=\underline{\fS_\la} \cdot{ d_\da}\cdot\underline{\msD_{\nu(\da)}\cap \fS_\mu}
=\underline{\fS_\la d_\da\fS_\mu}.$

\end{proof}

\begin{Example} Let $n=2$ and $r=3$. For the matrix
 $$\da=
\begin{pmatrix}
(0,0) & (1,0) \\
(1,1) & (0,0) \\
\end{pmatrix}\in {\Theta}_2(2,3)\quad\text{ with }\;|\da|=
\begin{pmatrix}
0 & 1 \\
2 & 0 \\
\end{pmatrix},$$
we have $\la=\ro(|\da|)=(1,2)$, $\mu=\co(|\da|)=(2,1)$, and $\nu(|\da|)=(0,2,1,0)$. Then $\fS_\la=\{1, s_2\}$,
$\fS_\mu=\{1,s_1\}$, $d_{|\da|}=s_1s_2$, and
$$\ddot{d}=d_{|\da|}d_{21}d_{12}=s_1s_2t_1t_3=s_0s_1s_0s_2.$$
But
$$T_{|\da|}T_{t_1}T_{t_3}=q^2T_{0102}+q(q-1)T_{10102}+(q-1)T_{1201012},$$
where we write $T_w=T_{i_1\cdots i_l}$ whenever $w=s_{i_1}\cdots s_{i_l}$ is reduced. Note that $d_\da=s_0s_1s_0s_2$ with $\nu(d_\da)=\nu(\da)=(0,0,1,1,1,0,0,0)$. By Proposition \ref{coincide 3},
$$\aligned
\frak b_{\da}&=q^{-2} x_\la T_{|\da|}T_{t_1}T_{t_3} x_\mu\\
&=x_\la T_{0102}x_\mu+q^{-1}(q-1)x_\la T_{10102}x_\mu+q^{-1}(q-1)x_\la T_{120102}x_\mu\\
&=T_{\fS_\la {d_\da}\fS_\mu}+q^{-1}(q-1)T_{\fS_\la {d'}\fS_\mu}+q^{-1}(q-1)T_{\fS_\la {d''}\fS_\mu},\endaligned$$
where $d'=s_1s_0s_1s_0s_2$ and $d''=s_1s_2s_0s_1s_0s_2$.
\end{Example}

\begin{Rem}\label{octah}
The slim cyclotomic $q$-Schur algebra $\tbsSr$ is isomorphic to the hyperoctahedral Schur algebras
(defined over $\scR=\mbz[q, q^{-1},Q,Q^{-1}]$ in indeterminates $q,Q$); see \cite[Defs 3.2.1, 4.3.3]{Gr97}
and \cite[Lem 4.3.4]{Gr97}. Note that our labelling $0,1,2,\ldots, r-1$ of the type B/C Dynkin diagram
corresponds to their labelling $r,\ldots,2,1$. The isomorphism can be induced from this correspondence.
By \cite[Rem. 6.3]{BKLW}, the type $C$ Schur algebra $\Si$ geometrically defined in \cite[\S6.1]{BKLW}
is isomorphic to the hyperoctahedral Schur algebra. Hence, $\tbsSr$ is isomorphic to
$\Si$.
\end{Rem}

\section{The cyclotomic Schur--Weyl duality: double centraliser property}

Let $C=(c_{ij})_{n\times n}$ be the Cartan matrix of affine type $A$; see, e.~g., \cite[(1.2.3.1)]{DDF}.

\begin{Def}Let $\bfU_\up(\afsl)$ be the quantum enveloping algebra over the field $\mathbb Q(\up)$ of
rational functions in indeterminate $\up$ associated with
the Cartan matrix $C$. Then it can be presented by the generators
$E_i,\ F_i,\  K_i,\ K_i^{-1}$ $(1\leq i\leq n)$ and the relations (QGL1)--(QGL5) in
\cite[Thm 2.3.1]{DDF}.\footnote{The quantum enveloping algebra $\bfU_\up(\afsl)$ here is called
the extended quantum affine $\mathfrak{sl}_n$ defined in \cite[p.~43]{DDF}. The non-extended version
is defined in \cite[\S1.3]{DDF}.} Define the quantum affine $\mathfrak{gl}_n$ by
$$\bfU_\up(\afgl)=\bfU_\up(\afsl)\otimes \mathbb Q(\up)[\sfc^+_1,\sfc_1^-, \sfc^+_2,\sfc_2^-,\ldots,\sfc_s^+,\sfc_s^-,\ldots],$$
where $\sfc_s^+,\sfc_s^-\; (s\in\mbz^+)$ are commuting indeterminates.
\end{Def}
The algebra $\bfU_\up(\afgl)$ is also a Hopf algebra with the structure maps $\Delta$, $\varepsilon$, and
$\sigma$ acting on the $\widehat{\mathfrak{sl}}_n$ part as usual plus the following extra action on the $\sfc_s^\pm$:
$$\Delta(\sfc_s^\pm)=\sfc_s^\pm\otimes 1+1\otimes\sfc^\pm_s,\quad \varepsilon(\sfc_s^\pm)=0,\quad\sigma(\sfc_s^\pm)=-\sfc_s^\pm.$$
This Hopf algebra is also isomorphic to the quantum loop algebra of $\mathfrak{gl}_n$, known as the
Drinfeld new realisation in \cite{Dr88}; see \cite[Thm 2.5.3]{DDF}.

The quantum group $\bfU_\up(\afgl)$ has a second construction in terms of a double Ringel--Hall algebra. This
new construction is crucial in the study of the (enhanced) affine quantum Schur--Weyl duality as presented in
\cite{DDF}. In this section, we investigate new applications of this theory to its cyclotomic counterpart.
We first briefly review the definition of the double Ringel--Hall algebra associated with a cyclic quiver.

 Let $\triangle=\triangle(n)$ ($n\geq 2$) be
the cyclic quiver with vertex set $I=\mbz/n\mbz$ and arrow set $\{i\longrightarrow i+1\}_{i\in I}$.
 Let $\Hall=\text{span}\{u_A\mid A\in\Theta_{\vtg}^+(n)\}$ be the Ringel--Hall algebra over $\sZ:=\mathbb Z[\up,\up^{-1}]$
 associated with $\triangle(n)$ and let
$\boldsymbol{\mathfrak H}_\vtg(n)=\Hall\otimes_\sZ \mbq(\up)$.
There are extended Ringel--Hall algebras over $\mathbb Q(\up)$
$$\Hallpi=\boldsymbol{\mathfrak H}_\vtg(n)\otimes \mbq(\up)[K_1^{\pm1},\ldots,K_n^{\pm1}]$$ and, similarly, $\Hallmi$
which possess the Green--Xiao Hopf algebra structure.
By considering a certain skew--Hopf pairing $\psi:\Hallpi\times\Hallmi\to\mathbb Q(\up)$, one defines
the Drinfeld double $\widehat\dHallr$ as the quotient algebra of the free product $\Hallpi*\Hallmi$ by an ideal defined by $\psi$; see \cite[\S2.1]{DDF}.
The {\it double Ringel--Hall algebra} is a reduced version of $\widehat\dHallr$:
$$\dHallr=\widehat\dHallr/{\scr I}\cong \boldsymbol{\mathfrak H}^+_\vtg(n)\otimes \mbq(\up)[K_1^{\pm1},\ldots,K_n^{\pm1}]
\otimes\boldsymbol{\mathfrak H}^-_\vtg(n)$$
 where $\scr I$ denotes the ideal generated by $1\otimes K_\al-K_\al\ot 1$ for
all $\al\in\mbz I$ and ${\mathfrak H}^+_\vtg(n)={\mathfrak H}_\vtg(n)$, ${\mathfrak H}^-_\vtg(n)={\mathfrak H}_\vtg(n)^{\text{op}}$.

Each vertex $i\in I$ of $\tri$ defines a simple representation which in turn defines generators $u^+_i$
and $u^-_i$.  There are so-called Schiffmann--Hubery generators $\sfz_s^+$ and $\sfz_s^-$ for $\dHallr$ constructed in \cite[\S2.3]{DDF}.
By \cite[Thm 2.3.1]{DDF}, there is a Hopf algebra isomorphism from $\bfU_\up(\afgl)$ to $\dHallr$ given by
$$E_i\longmapsto u^+_i,\; F_i\longmapsto u^-_i,\; K_i^\pm\longmapsto K_i^\pm,\;\sfc_s^\pm\longmapsto\sfz_s^\pm.$$
We identify the two algebras under this isomorphism.

Let $\Og$ be a free $\sZ$-module with basis $\{\og_i\mid
i\in\mbz\}$ and let $\bfOg=\Og\otimes_\sZ\mathbb Q(\up)$.
Using the action given in \cite[(3.5.0.1)]{DDF}, $\bfOg$ becomes a left $\dHallr$-module. Hence, the Hopf algebra
structure on $\dHallr$ induces a $\dHallr$-module structure on $\bfOg^{\otimes r}$. On the other hand,
$\Og^{\otimes r}$ has a right $\afsHr_\sZ$-module structure; see \cite[(3.3.0.4)]{DDF}, commuting with the left
$\dHallr$-module structure. Thus, we obtain a $\mathbb Q(\up)$-algebra homomorphism
\begin{equation}\label{xr}
\xi_r:\dHallr\lra\text{End}_{\afsHr_{\mathbb Q(\up)}}(\bfOg^{\otimes r})={\afsSr_{\mathbb Q(\up)}}.
\end{equation}

Let $\fD_\vtg(n)_\sZ^0$ denote the $\sZ$-subalgebra of $ \dHallr$ generated by $K_i^{\pm1}$ and
$\leb{K_i;0\atop t}\rib$ for $i\in I$ and $t>0$ and put $\fD_\vtg(n)_\sZ^\pm={\mathfrak H}_\vtg(n)^\pm$;
see \cite[\S2.4]{DDF} for the details.

\begin{Thm}Maintain the notations introduced above.
\begin{itemize}
\item[(1)] The $\sZ$-submodule
$$\fD_\vtg(n)_\sZ:=\Hall^+\fD_\vtg(n)^0\Hall^-\cong\Hall^+\ot\fD_\vtg(n)^0\ot\Hall^-$$
 is a Hopf $\sZ$-subalgebra of $ \dHallr$.
\item[(2)] There is a $\sZ$-algebra isomorphism $\afsSr_\sZ\cong \text{End}_{\afsHr_\sZ}(\Og^{\otimes r})$.
\item[(3)] The restriction of $\xi_r$ to $\fD_\vtg(n)_\sZ$ induces a $\sZ$-algebra epimorphism
$$\xi_r:\fD_\vtg(n)_\sZ\lra\text{End}_{\afsHr_\sZ}(\Og^{\otimes r}).$$
\item[(4)]
The map $\xi_r$ induces an $\scR$-algebra epimorphism
$$\xi_{r,\scR}:\fD_\vtg(n)_\sZ\otimes\scR\lra\sS_\vtg(n,r)_\scR.$$

\end{itemize}
\end{Thm}
Note that (1) is done in \cite{DF17}; (2) is given in \cite[Prop.~3.3.1]{DDF}; (3) can be derived from
 \cite[Thm 3.7.7]{DDF} and the proof of \cite[Thm 3.8.1]{DDF}, and (4) follows from (2).
 Note also that an explicit action of the generators of $\fD_\vtg(n)_\sZ$ on $\Og$ can be found in  \cite[Prop.~3.5.3]{DDF}.

We now extend the algebra homomorphism to slim cyclotomic $q$-Schur algebras. Recall the canonical
epimorphism $\ep_\bfu:\afsHr\ra \sH_\bfu(r)$ given in \eqref{epimor-af-AK-alg} under the assumption
that all the $u_i\in\scR$ in $\mathbf m$ are invertible. Regarding $\cysHr$ as an $\afsHr$-module via this map,
 we obtain the canonical right $\cysHr$-module isomorphism
$$\sT_\vtg(n,r)\ot_{\afsHr}\sH_\bfu(r)=\bigoplus_{\la\in\La(n,r)}x_\la \afsHr\ot_{\afsHr}\sH_\bfu(r)\cong
\bigoplus_{\la\in\La(n,r)}x_\la \sH_\bfu(r)=\sT_\bfu(n,r).$$
Thus, we obtain an $\scR$-algebra homomorphism
$$\ti\ep_\bfu:\afsSr\lra \sS_\bfu(n,r),\;f\lm f\ot\id.$$

 \begin{Prop} \label{compatibility-quotients}
Suppose that all the $u_i\in\scR$ are invertible. The $\scR$-algebra homomorphism $\ti\ep_\bfu$ is surjective.
In particular, the map $\xi_{r,\scR}$  extends to the $\scR$-algebra epimorphism
$$\ti\ep_\bfu\circ\xi_{r,\scR}:\fD_\vtg(n)_\sZ\otimes\scR\lra\cysSr_\scR.$$
\end{Prop}

 \begin{proof} Consider the linearly independent elements $\Phi_\da^\vtg$, $\da\in\cnrm$, for $\afsSr$,
 which are defined similarly to the basis elements $\Phi_\da$ with all $L_i$ replaced by $X_i$; see Remarks \ref{aff1}, \ref{aff2}
 and \ref{aff3}. By definition, we see easily that $\ti\ep_\bfu(\Phi_\da^\vtg)=\Phi_\da$. In other words,
 the images of these elements form a basis for $\cysSr$. Hence, $\ti\ep_\bfu$ is surjective.
 \end{proof}

 \begin{Rems} (1) The surjectivity of $\ti\ep_\bfu$ was given in \cite[Prop.~5.7(a)]{LR}.

 (2) For any $t\geq0$, let $ \dHallr^{(t)}=\bfU_\up(\afsl)\otimes \mathbb Q(\up)[\sfc^+_1,\sfc_1^-, \sfc^+_2,\sfc_2^-,\ldots,\sfc_t^+,\sfc_t^-]$. This is a Hopf subalgebra of $ \dHallr$.
 In fact, by \cite[Thm 3.8.1(2)]{DDF}, if $t$ satisfies $r=tn+t_0$, where $0\leq t_0<n$, then the
 restriction of the map $\xi_r$ in \eqref{xr} gives an epimorphism
 $$\xi_r^{(t)}:\dHallr^{(t)}\lra{\afsSr_{\mathbb Q(\up)}}.$$
  This together with $\ti\ep_\bfu$ induces an epimorphism
$$\xi_{r,\bfu}^{(t)}:\dHallr^{(t)}\lra{\sS_\bfu(n,r)_{\mathbb Q(\up)}}.$$
It would be interesting to conjecture that there exists a $t$, independent of $r$, such that
$\xi_{r,\bfu}^{(t)}$ is surjective for all $r\geq 0$.

Note that, if $\bfU_\up(\mathfrak{gl}_n)$ denotes the subalgebra generated by $E_i,F_i, K_j^{\pm1}$
for $1\leq i,j\leq n,i\neq n$, then $\dHallr^{(t)}$ contains the subalgebra
$$\bfU^{(t)}_\up(\mathfrak{gl}_n):=\bfU_\up(\mathfrak{gl}_n)\otimes
\mathbb Q(\up)[\sfc^+_1,\sfc_1^-, \sfc^+_2,\sfc_2^-,\ldots,\sfc_t^+,\sfc_t^-].$$
If the conjecture were true and the restriction of $\xi_{r,\bfu}^{(t)}$ to $\bfU^{(t)}_\up(\mathfrak{gl}_n)$
remained surjective, then the new object $\bfU^{(t)}_\up(\mathfrak{gl}_n)$ would be called
the {\it cyclotomic quantum $\mathfrak{gl}_n$.}

(3) We will prove in \cite{DDY} that if $\cysHr$ has a semisimple bottom in the sense of \cite{DR2},
then the algebra homomorphism
$$\xi_{r,\bfu}^\vee:\cysHr\lra\End_{\cysSr}(\Omega^{\otimes r})$$
is surjective. So it is natural to conjecture that the cyclotomic double centraliser property holds in general.
 \end{Rems}

\section {The cyclotomic Schur--Weyl duality: Morita equivalence}

As part of the Schur--Weyl duality, it is well-known that, for any field $\scK$, the $q$-Schur algebra
$\sS_q(n,r)_\scK$ is Morita equivalent to the Hecke algebra $\sH_q(r)_\scK$ if $n\geq r$ and $q$ is not
a root of the Poincar\'e polynomial of $\fS_r$. We now establish a similar result for the cyclotomic counterpart.

As in previous sections, $\scR$ is a commutative ring and $q\in \scR$ is invertible. Recall the notations $\cysHr=\cysHr_\scR$ and  $\cysSr=\cysSr_\scR$ and
from Proposition \ref{q-Sch} that $\sS_q(n,r)$ is (isomorphic to) a subalgebra of $\cysSr$.

Every $\lambda=(\la_1, \ldots, \la_n)\in \Lambda(n,r)$ gives a diagonal matrix $D_\la=\diag(\la)$ which in turn
defines a matrix $D_\la^{(m)}\in\cnrm$ via the embedding $\iota^{(m)}$ in \eqref{iota}. Then
$\Phi_{D_\la^{(m)}}=\frak l_\la$ is the idempotent map defined in \eqref{1la}. Note that all $\mathfrak l_\la$
live in the $q$-Schur subalgebra $\sS_q(n,r)$ of $\cysSr$; see Proposition \ref{q-Sch}.
In particular, $\frak l_\la^2=\frak l_\la$, $\frak l_\la\frak l_\mu=\delta_{\la, \mu}\frak l_\la$,
for $\la, \mu\in \Lambda(n,r)$ and
\begin{equation}\label{111}
 \sum_{\la\in \Lambda(n,r)}\frak l_\la=1.
 \end{equation}
Moreover, like the $q$-Schur algebra case, we have, for each $\da\in\cnrm$, the following relations in $\sS_\bfu(n,r)$

$$\Phi_{\da}\frak l_\la=\begin{cases} \Phi_{\da},\;\;  \text{if}\; \la=\col(|\da|);\\
0,\;  \;\; \;\;\text{otherwise }
\end{cases}\;\;\text{and}\;\;
\frak l_\la\Phi_{\da}=\begin{cases} \Phi_{\da},\;\;  \text{if}\; \la=\row(|\da|);\\
0,\;  \;\; \;\;\text{otherwise }.
\end{cases}
$$

{\it In the rest of the section, we assume $n\geq r$ and let $\omega=(\underbrace{1, \ldots, 1}_r,0, \ldots, 0)\in \Lambda(n,r)$.}
Then $\cysHr$ is a centraliser algebra of $\cysSr$ and \eqref{mumu} becomes
$$\frak l_\omega\sS_\bfu(n,r)\frak l_\omega\cong \cysHr.$$

We now consider some particular basis elements given in Theorem \ref{thm-standard-basis}.
For $\la,\mu\in \Lambda(n,r)$, let $\Phi_{\la, \mu}=\Phi_{A^{(m)}}\in\sS_\bfu(n,r)$, where $A=\th(\la,1,\mu)$
as defined in \eqref{theta}. In other words,
$$A=(|R_i^\la\cap R_j^\mu|)_{1\leq i,j\leq n}$$
where $R_i^\la, R_j^\mu$ are defined in \eqref{partial sum}. Thus, by definition,
\begin{equation}\label{particular elements}
\Phi_{\la, \mu}(x_\nu h)=\delta_{\nu, \mu}x_\la\sum_{w\in \msD_{\la\cap\mu}\cap \fS_\mu}T_w h,\quad\forall h\in \cysHr,
 \end{equation}
where $\fS_{\la\cap\mu}=\fS_\la\cap \fS_\mu$. In particular, if $\la=\mu=\nu$, then
\begin{equation}\label{special elements}\Phi_{\la, \la}(x_\la)=x_\la=\frak l_\la(x_\la)\text{ or }\Phi_{\la,\la}=\frak l_\la.
\end{equation}

For each $\la=(\la_1,\ldots,\la_n)\in\Lanr$, let $$P_\la(q)=P_{\fS_\la}(q)=\sum_{w\in \fS_\la}q^{\ell(w)}.$$
Then $x_\la^2=P_\la(q) x_\la$ for $\la\in \Lambda(n,r)$. Note that every $P_\la(q)$ is a factor of $P_{\fS_r}(q)$.

\begin{Prop}\label{Morita-equivalence} Keep the notations above and suppose that $n\geq r$.
If $P_{\fS_r}(q)$ is invertible in $\scR$, then $\sS_\bfu(n,r)$ and $\cysHr$
 are Morita equivalent.
\end{Prop}

\begin{proof} We apply the following general fact: for a ring $A$ together with an idempotent $e\in A$, $A$ and
 $eAe$ are Morita equivalent if and only if $AeA=A$; see \cite[Cor.~1.10]{Buch}.

Since ${\frak l}_\omega \sS_\bfu(n,r){\frak l}_\omega\cong \cysHr$, it suffices to prove that \begin{equation}\label{Mori}
\sS_\bfu(n,r){\frak l}_\omega\sS_\bfu(n,r)=\sS_\bfu(n,r).
\end{equation}
 By \eqref{111}, we only need to show ${\frak l}_\lambda=\Phi_{\la, \la}\in S_\bfu(n,r){\frak l}_\omega \sS_\bfu(n,r)$
for each $\lambda\in \Lambda(n,r)$.
%By \eqref{special elements}, $$\Phi_{\la, \la}(x_\la)=x_\la. $$
 Since $\fS_\omega=\{1\}$ and $x_\omega=1$, by \eqref{particular elements}, we have
$$\begin{aligned}
 \Phi_{\la, \omega}\Phi_{\omega, \la}(x_\la)&=
\Phi_{\la, \omega}(x_\omega\sum_{w\in\fS_\la}T_w)=
\Phi_{\la, \omega}(x_\omega\cdot x_\la)\\&= x_\la\cdot x_\la
=x_\la^2=P_\la(q) x_\la=P_\la(q) \Phi_{\la,\la}(x_\la).
\end{aligned}$$
This implies that
\begin{equation}\label{poincare}
\Phi_{\la, \omega}\Phi_{\omega, \la}=P_\la(q)\Phi_{\la, \la}.
\end{equation}
 Hence, by the hypothesis on $P_\la(q)$, we obtain that
$$\frak l_\la=\Phi_{\la, \la}=P_\la(q)^{-1}\Phi_{\la, \omega}\Phi_{\omega, \la} \in \sS_\bfu(n,r){\frak l}_\omega \sS_\bfu(n,r).$$
This finishes the proof.
\end{proof}

\begin{Rems} (1) The proposition above is a (slim) cyclotomic analogue of \cite[Appendix, Lem.~1.4]{DY}
which states that if $n\geq r$ and $P_{\fS_r}(q)$ is invertible in $\scR$, then $\afsSr_{\!\scR}$ is
Morita equivalent to $\afsHr_{\!\mathscr R}$. Note that this fact is proved
through a category equivalence in \cite[Thm 4.1.3]{DDF} under the assumption that $q$ is not a root of unity.

(2) For the cyclotomic $q$-Schur algebra, this Morita equivalence does not seem to be established in the literature.
It is easy to see that the proof above does not work in this case, since, for the idempotent $\Phi_{\la,\la}$ associated
with a multipartition $\la=(\la^{(1)},\ldots,\la^{(m)})$ of $r$, formula \eqref{poincare} can not be established
whenever $|\la^{(m)}|<r$.
\end{Rems}

Let $\cysHr\text{-Mod}$ (resp., $\cysSr\text{-Mod}$) be the category of left $\cysHr$-modules (resp., $\cysSr$-modules).
By Proposition \ref{Morita-equivalence}, we immediately have the following.

\begin{Coro}Let $\scR$ be a commutative ring in which $P_{\fS_r}(q)$ is invertible and assume $n\geq r$.
Then the categories $\cysHr\text{\rm -Mod}$ and $\cysSr\text{\rm -Mod}$ are equivalent.
\end{Coro}
By a standard argument for Morita equivalence, the category
 equivalence is given by the following two functors
 \begin{equation}\label{f1}
 \aligned
\mathcal{F}&:\sH\text{-Mod}\longrightarrow \sS\text{-Mod}, \;\;\;N\longmapsto  \sS e\otimes_{\sH}N,\\
\mathcal{G} &:\sS\text{-Mod}\longrightarrow \sH\text{-Mod}, \;\;\;M\longmapsto  e\sS\otimes_{\sS }M,\endaligned
\end{equation}
where $e=\mathfrak l_\omega$, $\sS=\cysSr$ ($n\geq r$), and $\sH=\cysHr=e\sS e$.

Let $\scK$ be a field and let $\cysHr_\scK\text{-\sf mod}$ (resp., $\cysSr_\scK\text{-\sf mod}$) be
the category of finite dimensional left $\cysHr_\scK$-modules (resp., $\cysSr_\scK$-modules).
 Then the functors above induce a category equivalence between $\cysHr_\scK\text{-\sf mod}$ and $\cysSr_\scK\text{-\sf mod}$.
 In the next section, we will describe the simple objects in $\cysSr_\scK\text{-\sf mod}$.

\section{Irreducible objects in $\cysSr_\scK\text{-\sf mod}$}

Throughout this section, let $\scK$ be a field and let $0\neq q\in \scK$. We now use the Morita equivalence
in the previous section and the classification of irreducible  $\cysHr_{\!\scK}$-modules to obtain
a classification of irreducible $\cysSr_{\!\scK}$-modules.

The classification of irreducible  $\cysHr_{\!\scK}$-modules is well known. In \cite{GL, DJM}, it is done
by using a cellular basis, while more precise labelling in terms of Kleshchev multipartitions is given
in \cite{AM,Ar2001}. We now have a brief review.

Recall that a partition $\la$ of $r$ is a
sequence of non-negative integers $\la_1\geq \la_2\geq \cdots$ such that $|\la|= \sum_i\la_i=r$.
By definition, an $m$-fold multipartition of $r$ (or simply, an $m$-multipartition) is a sequence
of $m$ partitions $\la=(\la^{(1)}, \ldots, \la^{(m)})$ such that $\sum_{1\leq i\leq m}|\la^{(i)}|=r$.
We use $\mathcal{P}_m(r)$ to denote the set of $m$-multipartitions of $r$.

The dominance order on $\mathcal{P}_m(r)$ is a partial order $\unrhd$ defined by setting,
for $\la,\mu\in \mathcal{P}_m(r)$, $\la\unrhd\mu$ if, for all $j, k$,
$$\sum_{1\leq l\leq k-1}|\la^{(l)}|+\sum_{1\leq i\leq j}\la^{(k)}_i\geq \sum_{1\leq l\leq k-1}|\mu^{(l)}|+\sum_{1\leq i\leq j}\mu^{(k)}_i.$$

To each  $\la\in \mathcal{P}_m(r)$ we can associate a left $\cysHr_\scK$-module $S^\la= \cysHr_\scK \mathtt z_\la$,
where $\mathtt z_\la=\tau(z_\la)\in\cysHr_\scK$ with $\tau, z_\la$ defined in \eqref{tau} and \cite[(2.1),(2.9)]{DR3}
(cf. \cite[Def. 3.28]{DJM}).  These modules are called Specht modules. Each Specht
module is naturally equipped with a bilinear form. Set $D^\la=S^\la/\text{rad}S^\la$, where
rad$S^\la$ is the radical of the bilinear form. By the cellular algebra theory, those nonzero $D^\la$ form a complete set of irreducible $\cysHr_\scK$-modules.

By determining Kashiwara's crystal graphs, Ariki and Mathas gave in \cite{AM} another classification of
irreducible $\cysHr_\scK$-modules in terms of  Kleshchev multipartitions. Let $\mathcal{KP}_m(r)\subseteq \mathcal{P}_m(r)$
denote the subset of Kleshchev multipartitions\footnote{See \cite{AM} for the definition. If $q$ is not a root of unity, and the parameters $u_1,\ldots,u_m$ are powers of $q$, see, e.g., \cite[3.5.4]{DW1} for a combinatorial definition.} of $r$. We now summarise these results in  the following theorem.

\begin {Thm}\label{s1}
Suppose that $ \scK$ is a field with $q,u_1, \ldots, u_m$ all nonzero.
\begin{itemize}
\item[(1)]{\rm(}\cite[Thm 3.30]{DJM}{\rm)}
 The set $\{D^\la\;|\; \la \in \mathcal{P}_m(r),D^\la\neq0 \}$
forms a complete set of non-isomorphic irreducible left $\cysHr_\scK$-modules. Further,
these modules are absolutely irreducible.

\item[(2)]{\rm(}\cite[Thm C]{AM}, \cite{Ar2001}{\rm)} $D^\la\neq 0$ if and only if $\la\in \mathcal{KP}_m(r)$.

\item[(3)] Let $\la, \mu$ be $m$-multipartitions of $r$ with $\mu\in \mathcal{KP}_m(r)$. If $[S^\la: D^\mu]\neq 0$, then $\la\unrhd \mu$. In particular, $[S^\la: D^\la]=1$ for all $\la\in \mathcal{KP}_m(r)$.
\end{itemize}
\end{Thm}

 The map $\tau$ defined in \eqref{tau} can be extended to an anti-automorphism of  $\cysSr$ (see \cite{DJM}).
We may turn a module over these algebras to its opposite side module. For example, we may twist the $(\cysSr,\cysHr)$-bimodule structure on $\sT_\bfu(n,r)$ into the $(\cysHr,\cysSr)$-bimodule $\sT_\bfu(n,r)^\tau$ with the action $h*x*s=\tau(s)x\tau(h)$ for all $h\in \cysHr, x\in \sT_\bfu(n,r),s\in\cysSr.$ Note also that there is an obvious left $\cysHr$-module isomorphism
\begin{equation}\label{Ttau}
\sT_\bfu(n,r)^\tau\cong\bigoplus_{\la\in\Lambda(n,r)}\cysHr x_\la.
\end{equation}

For $\mu\in \mathcal{KP}_m(r)$,
let $$L(\mu)=\mathcal{T}_\bfm(n,r)\otimes_{\cysHr_\scK} D^\mu.$$
By Proposition \ref{Morita-equivalence} and Theorem \ref{s1}, we have the following Corollary.

\begin{Coro}\label{n>r}Suppose $P_{\fS_r}(q)\neq0$ in $\scK$.
If $n\geq r$, then $\{L(\mu)\;|\;\mu\in \mathcal{KP}_m(r)\}$
is a complete set of non-isomorphic irreducible $\cysSr$-modules.
Moreover, we have an $\cysSr_\scK$-module isomorphism
$$L(\mu)\cong \Hom_{\cysHr_\scK}(\mathfrak l_\omega\cysSr_\scK, D^\mu)$$
\end{Coro}
\begin{proof}We only need to prove the second assertion. Let $\sS=\cysSr_\scK$. Then, for $e=\mathfrak l_\omega$,
$e\sS e\cong\cysHr_\scK$. Since $\sS$ is Morita equivalent to $e\sS e$, the multiplication map $\sS e\otimes_{e\sS e}e\sS\to\sS$
is an $(\sS,\sS)$-bimodule isomorphism (see, e.~g., \cite[Prop.~1.9(2)]{Buch}). Then the following (left)
$\sS$-module\footnote{See, e.g., \cite[II,Prop.~3.5]{Jac} for various left module structures on the Hom-space.} isomorphisms hold:
\begin{equation}\label{Nn}
\aligned
L(\mu)&\cong\Hom_\sS(\sS,L(\mu))\cong\Hom_\sS(\sS e\otimes_{e\sS e}e\sS,L(\mu))\\
&\cong \Hom_{e\sS e}(e\sS,\Hom_\sS(\sS e, L(\mu))\cong \Hom_{e\sS e}(e\sS,D^\mu),
\endaligned
\end{equation}
since $\Hom_\sS(\sS e, L(\mu))\cong eL(\mu)\cong D^\mu$.
\end{proof}

Recall from \cite{DJM} for the notion of semi-standard $\la$-tableaux of type $\mu$.

\begin{Lem}\label{tableaux}
For an $m$-multipartition $\la=(\la^{(1)},\ldots,\la^{(m)})$ and a composition $\mu\in\La(n,r)$, if the set
$\sT^{ss}(\la,{}^\circ\!\mu)$ of semistandard $\la$-tableaux of type ${}^\circ\mu:=((0),\ldots,(0),\mu)$
is not empty, then $\la^{(1)}+\cdots+\la^{(m)}\unrhd\mu$.
\end{Lem}
\begin{proof}
If $S=(S^{(1)},\ldots,S^{(m)})\in\sT^{ss}(\la,{}^\circ\!\mu)$, then each $S^{(i)}$ is a semistandard
tableau of type, say $\mu^{(i)}$. Thus, $\la^{(i)}\unrhd\mu^{(i)}$ for all $i$. Since
$\mu^{(1)}+\cdots+\mu^{(m)}=\mu$, it follows that $\la^{(1)}+\cdots+\la^{(m)}\unrhd\mu$.
\end{proof}

 For a partition $\la$, we use $l(\la)$ to denote the length of $\la$. Set
$$\mathcal{KP}_m(n,r)=\{(\la^{(1)}, \ldots, \la^{(m)})\in \mathcal{KP}_m(r)\;|\; l(\la^{(i)})\leq n, 1\leq i\leq m\}.$$

\begin{Lem}\label{parameter}
Assume now that $q\in\scK$ is not a root of unity.
Suppose that $n<r$ and $\mu\in \mathcal{KP}_m(r)$. Then $\mathcal{T}_\bfm(n,r)\otimes_{\cysHr_\scK} D^\mu\neq 0$
if and only if $\mu\in \mathcal{KP}_m(n,r)$.
\end{Lem}

\begin{proof} Assume $N\geq r>n$ and consider the idempotent
$f=\sum_{\la\in\Lambda(n,r)}\mathfrak l_{\widetilde\lambda}\in \sS=\sS_\bfu(N,r)_\scK$,
where $\widetilde\lambda=(\la_1, \ldots, \la_n, 0, \ldots, 0)\in\Lambda(N,r)$.
Then $\cysSr_\scK\cong f\sS f$. By \eqref{Nn}, we have
\begin{equation}\label{d1}
L(\mu)\cong fL(\widetilde\mu)\cong \Hom_{e\sS e}(e\sS f,D^\mu)\cong\Hom_{\cysHr_\scK}(\sT_\bfu(n,r)^\tau,D^\mu).\end{equation}
Here we have used the $(\cysHr_{\!\scK},\cysSr_{\!\scK})$-bimodule isomorphism $e\sS f\cong \sT_\bfu(n,r)^\tau$.
By \cite[Cor. 4.15]{DJM}, for each $\la\in \Lambda(n,r)$, there is a Specht module filtration of $\cysHr_\scK x_\la$
such that $[\cysHr_\scK x_\la:S^\nu]$ equals the number of semistandard $\nu$-tableaux of type $^\circ\la$, where $\nu\in \mathcal{P}_m(r)$.
By Lemma \ref{tableaux}, if $S^\nu$ occurs in the filtration of $\cysHr_\scK x_\la$, then $\sum_{1\leq i\leq m}\nu^{(i)}\unrhd\la$
 and, consequently, $\nu\in \mathcal{KP}_m(n,r)$.

Thus, if $\mathcal{T}_\bfm(n,r)\otimes_{\cysHr_\scK} D^\mu\neq 0$, by $\eqref{d1}$ and \eqref{Ttau},
$D^\mu$ is a composition factor of $S^\nu$ for some $\nu\in \mathcal{KP}_m(n,r)$. Then,
by Theorem \ref{s1}(3), $\nu\unrhd \mu$. Since $\nu\in \mathcal{KP}_m(n,r)$, we have $\mu\in \mathcal{KP}_m(n,r)$.

Conversely, suppose that $\mu=(\mu^{(1)}, \ldots, \mu^{(m)})\in \mathcal{KP}_m(n,r)$. Then
the irreducible $\cysHr_{\!\scK}$-module $D^\mu=\text{hd}(S^\mu)$ by Theorem \ref{s1}.
Inflate $D^\mu$ to an irreducible $\afsHr_\scK$-module via \eqref{epimor-af-AK-alg} and, by \cite[Lem.~4.1.1]{DW1},
there exists a multisegment $\bfs=\bfs^\sfc_{\mu,\bfu}=(\sfs_1,\ldots,\sfs_t)$ consisting of column
residual segments of $\mu$ such that the simple $\afsHr_\scK$-module $V_\bfs$ associated with $\bfs$
is isomorphic to $D^\mu$. Since $\mu \in \mathcal{KP}_m(n,r)$, each of the $\sfs_i$'s has length at most $n$.
Now, by \cite[Thm 6.6]{DD}, we obtain a simple $\afsSr_\scK$-module $\sT_\vtg(n,r)\otimes_{\afsHr_\scK} V_\bfs$. Hence,
$$\aligned
0&\neq\sT_\vtg(n,r)\otimes_{\afsHr_\scK} V_\bfs\cong\sT_\vtg(n,r)\otimes_{\afsHr_{\!\scK}} \cysHr_{\!\scK}\otimes_{\cysHr_{\!\scK}}V_\bfs\\
&\cong\sT_\bfu(n,r)\otimes_{\cysHr_{\!\scK}}D^\mu,\endaligned$$
as desired.
\end{proof}

Now we are ready to prove the main result of this section.

\begin{Thm}\label{simple module}
Assume $q\in\scK$ is not a root of unity.
For arbitrary positive integers $n,r$, the following set
$$\{\mathcal{T}_\bfm(n,r)\otimes_{\cysHr_\scK} D^\mu\;|\;\mu\in \mathcal{KP}_m(n,r)\}$$
forms a complete set of non-isomorphic irreducible $\cysSr_{\!\scK}$-modules.
\end{Thm}

\begin{proof}
If $n\geq r$, then $\mathcal{KP}_m(n,r)=\mathcal{KP}_m(r)$ and the result can be seen in Corollary \ref{n>r}.

Now suppose now $n<r\leq N$. Then there exists an idempotent $f\in\sS_\bfu(N,r)_\scK$ such that
\begin{equation}\label{simple set}
\{f(\mathcal{T}_\bfm(N,r)\otimes_{\cysHr_\scK} D^\mu)\;|\;\mu\in \mathcal{KP}_m(N,r)=\mathcal{KP}_m(r)\}\setminus \{0\}\end{equation}
forms a complete set of non-isomorphic irreducible $\cysSr$-modules.
Since $$f(\mathcal{T}_\bfm(N,r)\otimes_{\cysHr_\scK} D^\mu)\cong \mathcal{T}_\bfm(n,r)\otimes_{\cysHr_\scK} D^\mu,$$
which is nonzero if and only if $\mu\in \mathcal{KP}_m(n,r)$
by Lemma \ref{parameter}. Hence, our assertion follows.
\end{proof}

\begin{Rem} The hypothesis that $q$ is not a root of unity in
Theorem \ref{simple module} is stronger than the hypothesis that $P_{\fS_r}(q)\neq0$ in $\scK$.
This is required in order to use \cite[Thm 6.6]{DD}. There should be a direct proof under the latter hypothesis.

We also note that when $q$ is not a root of unity the algebra $\cysHr_\scK$, and hence $\cysSr_{\!\scK}$, can still be non-semisimple.
\end{Rem}

\begin{appendix}
\section{Proof of Theorem \ref{standard-basis-thm}}
Let $\leq$ be the Bruhat order on $\fS_r$, that is, $w'\leq w$ if $w'$ is a subexpression of some reduced expression of $w$.
For any $y, w\in \fS_r$, there is an element $y\ast w\in \fS_r$ such that $\ell(y\ast w)\leq \ell(y)+\ell(w)$ and
\begin{equation}\label{mult}
T_yT_w=\sum_{z\leq y\ast w}f_{y, w,z}T_z,\end{equation}
where $f_z^{y, w}\in \scR$; see \cite[Prop.~4.30]{DDPW}.
Also, by Lemma \ref{commutator} and an induction on the length $\ell(w)$ of $w$,
we have taht, for $\bfa,\bfb\in \mbz_m^r$ and $w,y\in \fS_r$, there exist $f_{y, \bfb}^{\bfa, w}\in\scR$ such that
\begin{equation}\label{Basic0}
L^\bfa T_w=T_wL^{\bfa w}+\sum_{y<w, \bfb\in \mbz_m^r}f_{y, \bfb}^{\bfa, w} T_{y} L^\bfb.
\end{equation}

\begin{Lem}\label{basis change} The set
$$\sX=\bigcup_{\lambda, \mu\in \Lambda(n,r)}\{T_uT_dL^\bfa T_v\mid u\in \fS_\lambda, d\in \msD_{\lambda,\mu},
\bfa\in \mbz_m^r, v\in \msD_{\nu(d)} \cap \fS_\mu \}$$
forms an $\scR$-basis of $\cysHr $.
\end{Lem}

\begin{proof} Recall form Lemma \ref{pbw basis} that $\cysHr$ is a free $\scR$-module with basis
$$\{T_wL^{\bfa}\mid w\in\fS_r, \bfa\in\mbz_m^r\},$$
 which, by Lemma \ref{Basic}(2), has the same cardinality as $\mathcal{X}$.
Thus, it suffices to prove that each basis element $T_wL^{\bfa}$ of $\cysHr$
is an $\scR$-linear combination of elements in $\sX$. We proceed by induction
on the length $\ell(w)$ of $w$.

Take arbitrary $w\in\fS_r$ and $\bfa\in\mbz_m^r$.
If $\ell(w)=0$, then $T_wL^\bfa=L^\bfa\in \sX$. Suppose now $\ell(w)\geq 1$. By Lemma \ref{Basic},
 there are uniquely determined elements $u\in\fS_\lambda$, $d\in \msD_{\lambda,\mu}$
 and $v\in \msD_{\nu(d)}\cap \fS_\mu$ such that
$w=udv$ and $\ell(w)=\ell(u)+\ell(d)+\ell(v).$
 Applying \eqref{Basic0} and \eqref{mult} gives the equalities
 $$\aligned
T_wL^\bfa&=T_uT_d(T_v L^\bfa)\\
&=T_uT_d(L^{\bfa v^{-1}}T_v+\sum_{y< v,\, \bfb\in\mbz_m^r}f_{y,\bfb}^{\bfa,v}\, T_y L^\bfb)\\
&=T_uT_dL^{\bfa v^{-1}}T_v+\sum_{y< v,\, \bfb\in\mbz_m^r}f_{y,\bfb}^{\bfa,v}\, T_{ud}T_y L^\bfb\\
&=T_uT_dL^{\bfa v^{-1}}T_v+\sum_{y< v,\, \bfb\in\mbz_m^r}f_{y,\bfb}^{\bfa,v}\, \sum_{z\leq (u d)\ast y, \,\bfb\in\mbz_m^r}f_{ud,y,z}T_z L^\bfb.\\
\endaligned$$
 Since $y<v$, we obtain that
$$\ell(z)\leq \ell(u)+\ell (d)+\ell(y)<\ell(u)+\ell (d)+\ell(v)=\ell(w).$$
By induction, $T_wL^\bfa$ is an $\scR$-linear
combination of elements in $\sX$.
\end{proof}

\begin{proof}[\bf Proof of Theorem \ref{standard-basis-thm}.] Take $\da\in \Theta_m(n,r)_{\la, \mu}$ and write $d=d_{|\da|}$ for
notational simplicity.
By \eqref{comm-x-nu-sigma}, $x_{\nu(d)}\sigma^{\ddot\da}=\sigma^{\ddot\da}x_{\nu(d)}$.
This together with \eqref{element} implies that
$$\aligned
\frak b_{\da}&=x_\la T_{d} \sigma^{\ddot\da}\sum_{v\in \msD_{\nu(d)}\cap \fS_\mu}T_v\\
&=\sum_{u\in \msD^{-1}_{\nu(d^{-1})}\cap \fS_{\la}}T_uT_{d}x_{\nu(d)}
\sigma^{\ddot\da}\sum_{v\in \msD_{\nu(d)}\cap \fS_\mu}T_v\\
&=\sum_{u\in \msD^{-1}_{\nu(d^{-1})}\cap \fS_{\la}}T_uT_{d} \sigma^{\ddot\da}\big(x_{\nu(d)}
\sum_{v\in \msD_{\nu(d)}\cap \fS_\mu}T_v\big)\\
&=\sum_{u\in \msD^{-1}_{\nu(d^{-1})}\cap \fS_{\la}}T_uT_{d} \sigma^{\ddot\da}x_{\mu}
\in \cysHr x_{\mu}.
\endaligned$$
Hence, ${\frak b}_\da\in x_\la\cysHr\cap \cysHr x_{\mu}.$

\vspace{.3cm}\noindent
{\it Linear Independence of $\mathcal{B}_{\la,\mu}$.}
Suppose that $$\sum_{\da\in \Theta_m(n,r)_{\la, \mu}}\gamma_{\da} \frak b_{\da}=0,
\text{ where $\gamma_{\da}\in\scR$.}$$
 By the definition, each $\frak b_{\da}$ is an $\scR$-linear combination
of elements of the form \linebreak $T_uT_{d_{|\da|}}L^\bfa T_v$ for $u\in \fS_\lambda$,
$\bfa\in \mbz_m^r$, and $v\in \msD_{\nu(d_{|\da|})} \cap \fS_\mu$.
Since, by Lemma \ref{from da to a}(1),
$$\cnrm_{\la, \mu}=\bigcup_{d\in\msD_{\la, \mu}}\cnrm_{\la, \mu}^d,$$
it follows from Lemma \ref{basis change} that for each fixed $d\in \dlm$,

\begin{equation}\label{form 6}
 \sum_{\da\in \Theta_m(n,r)_{\la, \mu}^d}\gamma_{\da} \frak b_{\da}=0.
 \end{equation}

We now fix such $d\in \dlm$. Let $d_0$ be the (unique) longest element
in $\msD_{\nu(d)}\cap \fS_\mu$ such that $w^0_\mu=w_{\nu(d)}^0d_0$,
where $w^0_\mu$ (resp., $w_{\nu(d)}^0$) denotes the longest
element of $\fS_\mu$ (resp., $\fS_{\nu(d)}$).
If $\da\in \cnrm_{\la, \mu}^d$, then, by \eqref{Basic0} and \eqref{mult},
 $$\aligned
 \frak b_{\da}&=T_{w_{\la}^0}T_d\, \sigma^{\ddot\da} T_{d_0}+\sum_
 {u\in\fS_{\la}, \, v\in \msD_{\nu(d)}\cap \fS_{\mu}\atop u< w^0_\la \,\text {or } \,
 v< d_0 }T_uT_d\,\sigma^{\ddot\da} T_v\\
 &= T_{w_{\la}^0}T_d\, \sigma^{\ddot\da} T_{d_0}+{\frak f}_{\da},
 \endaligned$$
where $\frak f_{\da}$ is an $\scR$-linear combination of the elements
$T_uT_d\,L^\bfa T_v$ with $\bfa\in\mbz_m^r$, $u\in \fS_\la$ and $v\in \msD_{\nu(d)}\cap \fS_{\mu}$
satisfying
$\ell(udv)<\ell(w_{\la}^0 d d_0)=\ell(w_{\la}^0)+\ell(d)+\ell(d_0).$
Substituting $\frak b_{\da}$ in \eqref{form 6} gives that
$$\sum_{\da\in \Theta_m(n,r)_{\la, \mu}^d}\gamma_{\da} T_{w_{\la}^0}T_d\,
 \sigma^{\ddot\da} T_{d_0}+\sum_{\da\in \Theta_m(n,r)_{\la, \mu}^d}\gamma_{\da}{\frak f}_{\da} =0.$$
By Lemma \ref{basis change} again, we obtain that
\begin{equation*}
\sum_{\da\in \Theta_m(n,r)_{\la, \mu}^d}\gamma_{\da} T_{w_{\la}^0}T_d\, \sigma^{\ddot\da} T_{d_0}=0.
\end{equation*}
Since all $T_w$ are invertible in $\cysHr$, the above equality gives that
 $$\sum_{\da\in \ddot{\Theta}_m(n,r)_{\la, \mu}^d}\gamma_{\da} \,\sigma^{\ddot\da} =0.$$
 Applying Proposition \ref{linear independent 0} forces that $\gamma_{\da}=0$ for each $\da\in \Theta_m(n,r)_{\la, \mu}^d$.
We conclude that $\mathcal{B}_{\la,\mu}$ is linearly independent.

\vspace{.5cm}

\noindent
{\it Proof of the Spanning Condition.} We now prove  that $\mathcal B_{\la,\mu}$ spans $x_\la\cysHr\cap \cysHr x_\mu$. Take an arbitrary element $z\in x_\la\cysHr\cap \cysHr x_\mu$.
By Lemma \ref{basis change}, $z$ can be written as
$$z=\sum_{d, \bfa, v}\gamma_{(d, \bfa, v)}x_\la T_dL^\bfa T_v,$$
where $\gamma_{(d, \bfa, v)}\in \scR$, and the sum is taken over all  $d\in\msD_{\la, \mu}$,
$\bfa\in\mbz_m^r$, and $v\in \msD_{\nu(d)}\cap \fS_\mu$.
By applying \eqref{element},
$$\aligned
z&=\sum_{d\in\msD_{\la, \mu}}\bigg(\sum_{u\in \msD^{-1}_{\nu(d^{-1})}\cap \fS_{\la}}T_u\bigg)T_{d}
\sum_{\bfa, v}\gamma_{(d, \bfa, v)}x_{\nu(d)}L^\bfa T_v\\
&=\sum_{d\in\msD_{\la, \mu}}\bigg(\sum_{u\in \msD^{-1}_{\nu(d^{-1})}\cap \fS_{\la}}T_u\bigg)T_{d}z_d,\endaligned
$$
where
\begin{equation}\label{zd}z_d=\sum_{\bfa\in\mathbb Z_m^r, v\in \msD_{\nu(d)}\cap \fS_\mu}\gamma_{(d, \bfa, v)}x_{\nu(d)}L^\bfa T_v.\end{equation}

We claim that $z_d\in\cysHr x_\mu$ for all $d\in\msD_{\la, \mu}$. Indeed, since $z\in\cysHr x_\mu$, we have
by Lemma \ref{permutation mod} that $zT_k=qz$ for $s_k\in J_\mu$.
Thus,
$$\sum_{d\in\msD_{\la, \mu}}\sum_{u\in \msD^{-1}_{\nu(d^{-1})}\cap \fS_{\la}}T_uT_{d}(z_dT_k-qz_d)=0.$$
Further, by \eqref{Basic0}, we have for each fixed $d\in\msD_{\la, \mu}$,
$$z_d=\sum_{\bfa, v}\gamma_{(d, \bfa, v)}x_{\nu(d)}L^\bfa T_v
=\sum_{\bfa, v}\gamma_{(d, \bfa, v)}x_{\nu(d)}\big(T_{v}L^{\bfa v}+\sum_{y<v, \bfb\in\mbz_m^r}f^{\bfa,v}_{y,\bfb}T_{y}L^\bfb\big).$$
Since $\fS_{\nu(d)}\subseteq \fS_{\mu}$, $z_d$ can be written as an $\scR$-linear combination of
$T_wL^\bfa$ with $w\in \fS_\mu$ and $\bfa\in\mbz_m^r$.
Thus, for each $s_k\in J_\mu$, $z_dT_k$ can also be written as an $\scR$-linear combination of
$T_wL^\bfa$ with $w\in \fS_\mu$ and $\bfa\in\mbz_m^r$. By Lemmas \ref{Basic}(2) and \ref{pbw basis},
the elements $T_uT_dT_wL^\bfa$ for $u\in \msD^{-1}_{\nu(d^{-1})}\cap \fS_{\la}$, $w\in\fS_\mu$ and
$\bfa\in\mbz_m^r$ are linearly independent. This forces that $z_dT_k-qz_d=0$ for all
$s_k\in J_\mu$. By Lemma \ref{permutation mod}, $z_d\in\cysHr x_\mu$, proving the claim.

Let $d\in\msD_{\la, \mu}$. If $w\in \fS_{\nu(d)}$ and $w'<w$ (the Bruhat order of $\fS_r$), then $w'\in\fS_{\nu(d)}$.
By applying \eqref{Basic0}, we obtain that for each $d\in\msD_{\la, \mu}$ and $v\in \msD_{\nu(d)}\cap \fS_\mu$,
\begin{equation}\label{form1}
\sum_{\bfa\in\mbz_m^r}\gamma_{(d, \bfa, v)}x_{\nu(d)}L^\bfa
=\sum_{\bfb\in\mbz_m^r, w\in\fS_{\nu(d)}}h_{( \bfb, w)}L^\bfb T_w,
\end{equation}
where $h_{(\bfb, w)}\in\scR$. Thus, we obtain that
$$z_d=\sum_{\bfb\in\mbz_m^r,\, w\in\fS_{\nu(d)}}h_{( \bfb, w)}L^\bfb T_w(\sum_{v\in \msD_{\nu(d)}\cap \fS_\mu}T_v).$$
Further, for each $w\in\fS_{\nu(d)}$ and $v\in  \msD_{\nu(d)}\cap \fS_\mu$, we have $T_wT_v=T_{wv}$
with $wv\in\fS_\mu$. Then by the fact that $z_d\in\cysHr x_\mu$ and
Lemma \ref{permutation mod}, $z_d$ has the form
$$z_d=\sum_{\bfc\in\mbz_m^r}f_{(d, \bfc)}L^\bfc x_\mu$$
 for some $f_{(d, \bfc)}\in\scR$. By further applying \eqref{element}, we obtain that
$$z_d=\sum_{\bfc\in\mbz_m^r}f_{(d, \bfc)}L^\bfc x_\mu =\sum_{\bfc\in\mbz_m^r}f_{(d, \bfc)}L^\bfc x_{\nu(d)}\sum_{v\in\msD_{\nu(d)}\cap \fS_\mu}T_v.$$
In other words,
 $$\sum_{\bfa\in\mbz_m^r, v\in\msD_{\nu(d)}\cap \fS_\mu}\gamma_{(d, \bfa, v)}x_{\nu(d)}L^\bfa T_v=\sum_{\bfc\in\mbz_m^r}f_{(d, \bfc)}L^\bfc x_{\nu(d)}\sum_{v\in\msD_{\nu(d)}\cap \fS_\mu}T_v,$$
 that is,
 $$\sum_{v\in \msD_{\nu(d)}\cap \fS_\mu}\bigg(\sum_{\bfa\in\mbz_m^r}\gamma_{(d, \bfa, v)}x_{\nu(d)}L^\bfa -\sum_{\bfc\in\mbz_m^r}f_{(d, \bfc)}L^\bfc x_{\nu(d)}\bigg)T_v=0.$$
By \eqref{Basic0}, each term
$\sum_{\bfa\in\mbz_m^r}\gamma_{(d, \bfa, v)}x_{\nu(d)}L^\bfa -\sum_{\bfc\in\mbz_m^r}f_{(d, \bfc)}L^\bfc x_{\nu(d)}$
can be written as an $\scR$-linear combination of
 $L^\alpha T_w$ with $w\in \fS_{\nu(d)}$ and $\alpha\in\mbz_m^r$.
 By Lemma \ref{pbw basis}, the elements $L^\alpha T_wT_v$ for $\alpha\in\mbz_m^r$, $w\in \fS_{\nu(d)}$
 and $v\in\msD_{\nu(d)}\cap \fS_\mu$ are linearly independent. It follows that for each
 $v\in \msD_{\nu(d)}\cap \fS_\mu$,
$$\sum_{\bfa\in\mbz_m^r}\gamma_{(d, \bfa, v)}x_{\nu(d)}L^\bfa -\sum_{\bfc\in\mbz_m^r}f_{(d, \bfc)}L^\bfc x_{\nu(d)}=0,$$
and thus,
$$\sum_{\bfa\in\mbz_m^r}\gamma_{(d, \bfa, v)}x_{\nu(d)}L^\bfa =\sum_{\bfc\in\mbz_m^r}f_{(d, \bfc)}L^\bfc x_{\nu(d)}
\in x_{\nu(d)}\cysHr\cap \cysHr x_{\nu(d)}.$$
Then by Proposition \ref{symm-funct-composition}(2), $\sum_{\bfa}\gamma_{(d, \bfa, v)}x_{\nu(d)}L^\bfa$ is
an $\scR$-linear combination of elements in $\{x_{\nu(d)}\sigma^\bfe\mid \bfe\in \Gamma(m,\nu(d))\}$.
Applying Lemma \ref{xim} implies that
$\sum_{\bfa}\gamma_{(d, \bfa, v)}x_{\nu(d)}L^\bfa$ is an $\scR$-linear
combination of elements in $\{x_{\nu(d)}\sigma^{\ddot\da}\mid \da\in \cnrm_{\la, \mu}^{d}\}$.

We conclude that
$$\aligned
z&=\sum_{d\in\msD_{\la, \mu}}\bigg(\sum_{u\in \msD^{-1}_{\nu(d^{-1})}\cap \fS_{\la}}T_u\bigg)T_{d}\sum_{\bfa, v}
\gamma_{(d, \bfa, v)}x_{\nu(d)}L^\bfa T_v\\
&=\sum_{d\in\msD_{\la, \mu}}\bigg(\sum_{u\in \msD^{-1}_{\nu(d^{-1})}\cap \fS_{\la}}T_u\bigg)T_{d}
\sum_{\bfe, v}f_{(d, \bfe)}x_{\nu(d)}\sigma^\bfe T_v\\
&=\sum_{d\in\msD_{\la, \mu}}\sum_{\da\in \cnrm_{\la, \mu}^{d}}f_{(d, \da)}x_\la T_{d}\sigma^{\ddot\da}
\sum_{v\in\msD_{\nu(d)}\cap \fS_\mu}T_v\\
&=\sum_{d\in\msD_{\la, \mu}}\sum_{\da\in \cnrm_{\la, \mu}^{d}}f_{(d, \da)}\bfb_{\da}
\in\sum_{\da\in \cnrm_{\la, \mu}}\scR \frak b_{\da}.
\endaligned$$
This finishes the proof.
\end{proof}

\section{The affine version of Lemma \ref{sym}}

\begin{Lem}\label{afsym}
Suppose that $$z=\sum_{\bfb\in \mbz^r}c_\bfb x_{(r)}X^\bfb\in x_{(r)}\afsHrR\cap \afsHrR x_{(r)},
\;\text{ where $c_\bfb\in \scR$.}$$
Then $c_\bfb=c_{\bfb w}$ for all $\bfb\in \mbz^r$ and $w\in \fS_r$.
\end{Lem}

\begin{proof} For each $\bfa=(a_1,\ldots,a_r)\in\mbz^r$, set
$$\|\bfa\|={\rm max}\{|a_i|\mid 1\leq i\leq r\}.$$
 It is clear that $\|\bfa\|=\|\bfa w\|$ for each $w\in \fS_r$.
Since only finitely many $c_\bfb$ may not be zero, there is a positive integer $N$ such that
$$c_\bfb=0\;\;\text{whenever $\|\bfb\|>N$.}$$

It suffices to prove that $c_\bfa=c_{\bfa s_i}$ for each fixed $\bfa\in \mbz^r$ and $1\leq i<r$.
Write $\bfa=(a_1,\ldots,a_r)$. If $a_i=a_{i+1}$, then
$\bfa=\bfa s_i$ and hence, $c_\bfa=c_{\bfa s_i}$, as desired. Now let $a_i\not=a_{i+1}$.
We may suppose $a_i<a_{i+1}$ (Otherwise, we replace $\bfa$ by $\bfa s_i$). We proceed by induction
on $a_{i+1}$. If $|a_{i+1}|>N$, then $\|\bfa\|=\|\bfa s_i\|>N$. Thus, $c_\bfa=c_{\bfa s_i}=0$.
Let $-N\leq k\leq N$ and suppose that for each $\bfa$ with $k+1\leq a_{i+1}$, the
equality $c_\bfa=c_{\bfa s_i}$ holds. Now we take $\bfa\in\mbz^r$ with $a_i<a_{i+1}=k$.

By Lemma \ref{commutator}, we obtain that
 $$\aligned
 z T_i=&\sum_{\bfb\in \mbz^r}c_\bfb x_{(r)}X^\bfb T_i\\
  =&\sum_{\bfb\in \mbz^r}c_\bfb x_{(r)}T_i X^{\bfb s_i} +\sum_{\bfb\in \mbz^r\atop b_i< b_{i+1}}(q-1)c_\bfb x_{(r)}
  \sum_{t=1}^{b_{i+1}-b_i}X^{\bfb s_i-t\alpha_i}\\
 &+\sum_{\bfb\in \mbz^r\atop b_i> b_{i+1}}(1-q)c_\bfb x_{(r)}\sum_{t=0}^{b_{i}-b_{i+1}-1}X^{\bfb s_i+t\alpha_i}.\\
\endaligned$$

Since $z\in \afsHrR x_{(r)}$ and $x_{(r)}T_i=qx_{(r)}$, it follows that $qz=z T_i$ which
gives rise to the equality
$$\aligned
 \sum_{\bfb\in \mbz^r}q c_\bfb x_{(r)}X^\bfb=&\sum_{\bfb\in \mbz^r}q c_\bfb x_{(r)} X^{\bfb s_i}
+\sum_{\bfb\in \mbz^r\atop b_i< b_{i+1}}(q-1)c_\bfb x_{(r)}\sum_{t=1}^{b_{i+1}-b_i}X^{\bfb s_i-t\alpha_i}\\
 &\qquad\qquad\qquad \quad+\sum_{\bfb\in \mbz^r\atop b_i> b_{i+1}}(1-q)c_\bfb x_{(r)}
 \sum_{t=0}^{b_{i}-b_{i+1}-1}X^{\bfb s_i+t\alpha_i}.
\endaligned$$
 For the $\bfa\in\mbz^r$ chosen as above, comparing the coefficients of $x_{(r)}X^\bfa$ on both sides
 implies that
\begin{equation} \label{coefficients-comp-affine}
qc_\bfa=qc_{\bfa s_i}+c'+c'',
\end{equation}
 where $c'$ (resp., $c''$) denotes the coefficient of $x_{(r)}X^\bfa$ in the second
(resp., third) sum of the right hand side. In the following we calculate $c'$ and $c''$.

Consider those $\bfb=(b_1, \ldots, b_r)\in\mbz^r$ with $b_i< b_{i+1}$ such that
$$\bfb s_i-t\alpha_i=\bfa\;\text{ for some $1\leq t\leq b_{i+1}-b_i$.}$$
Then $b_i=a_{i+1}-t$ and $b_{i+1}=a_i+t$, and thus, $t\leq b_{i+1}-b_i=a_i-a_{i+1}+2t$. Hence,
$$a_{i+1}-a_i\leq t.$$
  Since $\bfa s_i-t\alpha_i=\bfa-(t-(a_{i+1}-a_i))\alpha_i$, by substituting
$t$ for $t-(a_{i+1}-a_i)$, we obtain that $\bfb=\bfa-t\alpha_i$
with $t\geq 0$.

Similarly, those $\bfb=(b_1, \ldots, b_r)\in\mbz_m^r$ with $b_i>b_{i+1}$ such that
$$\bfb s_i+t\alpha_i=\bfa\;\text{ for some $0\leq t\leq b_{i}-b_{i+1}-1$}$$
are simply $\bfb=\bfa s_i+t\alpha_i$ for all $t\geq 0$.

We conclude that
$$c'=(q-1)\sum_{t\geq 0}c_{\bfa-t\alpha_i}
\;\text{ and }\;
c''=(1-q)\sum_{t\geq 0}c_{\bfa s_i+t\alpha_i}.$$
  Hence, we obtain from \eqref{coefficients-comp-affine} that
$$\aligned
qc_\bfa&=qc_{\bfa s_i}+(q-1)\sum_{t\geq 0} c_{\bfa-t\alpha_i}
+(1-q)\sum_{t\geq 0} c_{\bfa s_i+t\alpha_i}\\
&=qc_{\bfa s_i}+(q-1)\big(c_\bfa+\sum_{t\geq 1} c_{\bfa-t\alpha_i}\big)
+(1-q)\big(c_{\bfa s_i}+\sum_{t\geq 1} c_{\bfa s_i+t\alpha_i}\big).
\endaligned$$
 This together with the fact that $\bfa s_i+t\alpha_i=(\bfa-t\alpha_i)s_i$ implies that
\begin{equation} \label{form2}
\aligned
c_\bfa-c_{\bfa s_i}=\sum_{t\geq 1}(q-1)\big(c_{\bfa-t\alpha_i}-c_{(\bfa-t\alpha_i)s_i}\big).\endaligned
\end{equation}
Since for each $t\geq 1$, $\bfa-t\alpha_i=:{\bf a}^{(t)}=(a_1^{(t)},\ldots,a^{(t)}_r)$ satisfies
$$a^{(t)}_i=a_i-t<a_{i+1}+t=a^{(t)}_{i+1}\;\text{ and }\;a^{(t)}_{i+1}=a_{i+1}+t\geq k+1,$$
 we have by the induction hypothesis that
$c_{\bfa-t\alpha_i}=c_{(\bfa-t\alpha_i)s_i}$ for all $t\geq 1$. Consequently, $c_\bfa=c_{\bfa s_i}$.
\end{proof}

\end{appendix}

\end{document}